\documentclass[12pt]{article}
\usepackage[T1]{fontenc}
\usepackage{amsmath}
\usepackage{amssymb}
\usepackage{amsthm}
\usepackage{tikz}
\usetikzlibrary{patterns}
\usepackage{pst-all}

\newtheorem{theorem}{Theorem}[section]
\newtheorem{definitio}[theorem]{Definition}
\newenvironment{definition}{\begin{definitio} \rm }{\end{definitio}}
\newtheorem{lemma}[theorem]{Lemma}
\newtheorem{proposition}[theorem]{Proposition}
\newtheorem{corollary}[theorem]{Corollary}
\newtheorem{sublemma}[theorem]{Sublemma}

\def\fl#1{\smash{\mathop{\hbox to 11mm{ \rightarrowfill\ }}\limits^{\textstyle #1}}}

\newcommand{\lbp}{\textrm{{\Large (}}}

\newcommand{\rbp}{\textrm{{\Large )}}}
\newcommand{\codim}{\textnormal{codim}}

\newcommand{\proofof}[1]{\noindent{{\em Proof of #1.}}}
\newcommand{\fin}{\hfill$\square$\bigskip}

\newcommand{\N}{\mathbb{N}}
\newcommand{\Z}{\mathbb{Z}}
\newcommand{\Q}{\mathbb{Q}}
\newcommand{\C}{\mathbb{C}}
\newcommand{\Qt}{\mathbb{Q}[t^{\pm1}]}
\newcommand{\Zt}{\mathbb{Z}[t^{\pm1}]}
\newcommand{\Qti}{\mathcal{R}}
\newcommand{\Ct}{\mathbb{C}[t^{\pm1}]}

\newcommand{\F}{\mathcal{F}}
\newcommand{\G}{\mathcal{G}}
\newcommand{\K}{\mathcal{K}}
\newcommand{\I}{\mathcal{I}}

\newcommand{\Sym}{\mathcal{S}}
\newcommand{\Ens}{\mathcal{P}}
\newcommand{\Al}{\mathfrak{A}}
\newcommand{\bl}{\mathfrak{b}}
\newcommand{\h}{\mathfrak{h}}
\newcommand{\hs}{\hslash}
\newcommand{\HA}{\mathcal{H}}
\newcommand{\Ah}{\mathfrak{A}_h}
\newcommand{\perm}{\mathcal{S}_3}
\newcommand{\sg}{\varepsilon(\sigma)}

\newcommand{\mi}{\underline{i}}
\newcommand{\mk}{\underline{k}}
\newcommand{\ml}{\underline{\ell}}
\newcommand{\md}{\underline{\delta_I}}
\newcommand{\NT}{\Xi}
\newcommand{\maps}{\mathbf{\Phi}}
\newcommand{\Rd}{\mathcal{R}}

\title{Equivariant triple intersections}
\author{Delphine Moussard\footnote{The author was supported by the Italian FIRB project 
"Geometry and topology of low-dimensional manifolds", RBFR10GHHH.}}
\date{}

\begin{document}

\maketitle

\begin{abstract}
Given a null-homologous knot $K$ in a rational homology 3-sphere $M$, and the standard infinite cyclic covering $\tilde{X}$ of $(M,K)$, 
we define an invariant of triples of curves in $\tilde{X}$, by means of equivariant triple intersections of surfaces. We prove that 
this invariant provides a map $\phi$ on $\Al^{\otimes 3}$, where $\Al$ is the Alexander module of $(M,K)$, and that the isomorphism 
class of $\phi$ is an invariant of the pair $(M,K)$. For a fixed Blanchfield module $(\Al,\bl)$, we consider pairs $(M,K)$ whose Blanchfield 
modules are isomorphic to $(\Al,\bl)$, equipped with a marking, {\em i.e.} a fixed isomorphism from $(\Al,\bl)$ to the Blanchfield module 
of $(M,K)$. In this setting, we compute the variation of $\phi$ under null borromean surgeries, and we describe the set of all maps $\phi$. 
Finally, we prove that the map $\phi$ is a finite type invariant of degree 1 of marked pairs $(M,K)$ with respect to null 
Lagrangian-preserving surgeries, and we determine the space of all degree 1 invariants of marked pairs $(M,K)$ with rational values.


\end{abstract}

\tableofcontents

    \section{Introduction}

In \cite{GR}, Garoufalidis and Rozansky introduced a theory of finite type invariants of knots in integral homology spheres with respect to the null-move -- 
the move which defines the Goussarov-Habiro theory of finite type invariants of 3-manifolds, with a nullity condition with respect to the knot. 
In particular, they proved that the Kricker lift of the Kontsevich integral constructed by Kricker \cite{Kri} (see also \cite{GK}) is a universal finite type invariant 
of knots in integral homology spheres with trivial Alexander polynomial. In \cite{Mt}, we extended this result to finite type invariants of null-homologous 
knots in rational homology spheres, with respect to a move called null Lagrangian-preserving surgery -- which generalizes the null-move to the setting 
of rational homology--, in the case of a trivial Alexander polynomial. We also studied the case of a non-trivial Alexander polynomial. 
The study of these theories of finite type invariants gives tools to understand the Kricker lift of the Kontsevich integral, and to compare it with other 
powerful invariants as the one constructed by Lescop \cite{Les2} by means of equivariant intersections in configuration spaces. 

In this paper, we construct and study an invariant of null-homologous knots in rational homology spheres, which appears to have finiteness properties 
with respect to null Lagrangian-preserving surgeries when a parametrization of the Alexander module -- a marking -- is fixed. Such a marking is preserved 
by null Lagrangian-preserving surgeries, hence the theory of finite type invariants can be defined for null-homologous knots in rational homology spheres 
with a fixed marking, and it provides a richer and more faithful theory. 

The Kricker invariant organizes the Kontsevich integral into a series of terms ordered by their loop degree -- given by the first Betti number of the graphs. 
As proved by Garoufalidis and Rozansky \cite[Corollary 1.5]{GR}, the $n$-loop part of this invariant is a finite type invariant of degree $2n-2$ with respect 
to the null-move. The invariant constructed in this paper takes place in some sense beetween the 1-loop part -- explicitly given by the Alexander polynomial 
\cite[Theorem 1.0.8]{Kri} -- and the 2-loop part -- which coincides with the triple equivariant intersection of Lescop \cite{Les4} at least for knots 
in integral homology spheres with trivial Alexander polynomial -- of the Kricker invariant, but it exists as a finite type invariant only when a marking 
of the Alexander module is fixed.

\paragraph{Description of the paper}
We consider pairs $(M,K)$, where $M$ is a {\em rational homology 3-sphere} ($\Q$HS), 
{\em i.e.} an oriented compact 3-manifold which has the same homology with rational coefficients as the standard 
3-sphere $S^3$, and $K$ is a {\em null-homologous knot} in $M$, {\em i.e.} a knot whose class in $H_1(M;\mathbb{Z})$ is trivial. 
We define an invariant of triples of curves in the associated infinite cyclic covering, by means of equivariant triple intersection 
numbers of surfaces. It provides a map 
$\phi$ on $\displaystyle \Al_h=\frac{\Al^{\otimes 3}}{(\otimes_{1\leq j\leq 3} \beta_j=\otimes_{1\leq j\leq 3} t\beta_j)}$, 
where $\Al$ is the Alexander module of $(M,K)$. The isomorphism class of $(\Al,\phi)$ is an invariant of the homeomorphism 
class of $(M,K)$. 

Then for a fixed Blanchfield module $(\Al,\bl)$, {\em i.e.} an Alexander module endowed with a Blanchfield form, 
we consider marked pairs $(M,K,\xi)$, where $\xi$ is an isomorphism 
from $(\Al,\bl)$ to the Blanchfield module of $(M,K)$. For such marked pairs, the map $\phi$ is well-defined, not only up to isomorphism. 
In this setting, we compute the variation of $\phi$ under the null-move of Garoufalidis and Rozansky \cite{GR}, called here null borromean 
surgery. As a consequence, we see that the equivariant triple intersection map $\phi$ is a finite type invariant of degree one 
of the marked pairs $(M,K,\xi)$ with respect to null borromean surgeries. 

For a fixed Blanchfield module $(\Al,\bl)$, we identify the rational vector space of all equivariant triple intersection 
maps $\phi$ of marked pairs $(M,K,\xi)$ with the space 
$\displaystyle \HA=\frac{\Lambda^3_\Q\Al}{(\beta_1\wedge\beta_2\wedge\beta_3=t\beta_1\wedge t\beta_2\wedge t\beta_3)}$. 
We study the vector space $\HA$, and give bounds for its dimension. 

In the last section, we consider null Lagrangian-preserving surgeries, a move which includes the null-move of Garoufalidis and Rozansky, 
and which is transitive on the set of marked pairs $(M,K,\xi)$ for a fixed Blanchfield module. We show that the map $\phi$ is a finite type 
invariant of degree one of the marked pairs $(M,K,\xi)$ with respect to null Lagrangian-preserving surgeries. We prove that 
the map $\phi$, together with degree one invariants obtained from the cardinality of $H_1(M;\Z)$, provides a universal 
rational valued degree one invariant of the marked pairs $(M,K,\xi)$ with respect to null Lagrangian-preserving surgeries. 
We obtain similar results in the case of pairs $(M,K,\xi)$ where $M$ is an integral homology 3-sphere, and the marking $\xi$ 
is defined on the integral Blanchfield module.

I wish to thank Christine Lescop for useful suggestions and comments.

\paragraph{Convention and notation}\ \\
The boundary of an oriented manifold is oriented with the outward normal first convention. 
We also use this convention to define the co-orientation of an oriented manifold embedded in another oriented manifold. \\
Unless otherwise mentioned, all tensor products and exterior products are defined over~$\Q$.  \\
The homology class of a curve $\gamma$ in a manifold is denoted by $[\gamma]$.  \\
For $n\in\N\setminus\{0\}$, $S^n$ is the standard $n$-dimensional sphere.  \\
If $C_1$, .., $C_k$ are transverse integral chains in a manifold $M$, such that the sum of the codimensions of the $C_i$ equals the dimension 
of $M$, $<C_1,..,C_k>_M$ is the algebraic intersection number of the $C_i$ in $M$. \\
For chains $C_1$ and $C_2$ in a manifold $M$, the transversality condition includes $\partial C_1\cap\partial C_2=\emptyset$.

    \section{Statement of the results}

  \subsection{Equivariant triple intersections} \label{subsecint}

We first recall the definition of the Alexander module. 
Let $(M,K)$ be a {\em $\Q$SK-pair}, {\em i.e.} a pair made of a rational homology 3-sphere $M$ and a null-homologous knot $K$ in $M$. 
Let $T(K)$ be a tubular neighborhood of $K$. The \emph{exterior} of $K$ is 
$X=M\setminus Int(T(K))$. Consider the projection $\pi : \pi_1(X) \to \frac{H_1(X;\mathbb{Z})}{torsion} \cong \mathbb{Z}$,
and the covering map $p : \tilde{X} \to X$ associated with its kernel. The covering $\tilde{X}$ is the \emph{infinite cyclic covering} 
of $X$. The automorphism group of the covering, $Aut(\tilde{X})$, is isomorphic to $\mathbb{Z}$. It acts on 
$H_1(\tilde{X};\mathbb{Q})$. Denoting the action of a generator $\tau$ of $Aut(\tilde{X})$ as the multiplication by $t$, 
we get a structure of $\Qt$-module on $\Al(M,K)=H_1(\tilde{X};\mathbb{Q})$. This $\Qt$-module is called the \emph{Alexander module} 
of $(M,K)$. It is a finitely generated torsion $\Qt$-module. We denote the annihilator of $\Al(M,K)$ by $\delta_{(M,K)}(t)$, 
normalized so that $\delta_{(M,K)}(t)\in\Q[t]$, $\delta_{(M,K)}(0)\neq0$ and $\delta_{(M,K)}(1)=1$. By a slight abuse of notation, for any 
$\Q$SK-pair, we denote by $\tau$ the automorphism of the infinite cyclic covering which induces the multiplication by $t$ 
in the Alexander module, and for a polynomial $P=\sum_{k\in\Z}a_kt^k\in\Qt$ and a chain $C$ in the infinite cyclic covering, 
we denote by $P(\tau)C$ the chain $\sum_{k\in\Z}a_k\tau^k(C)$. 

We aim at defining an equivariant triple intersection map on the rational vector space:
$$\Ah(M,K)=\frac{\Al(M,K)^{\otimes 3}}{(\beta_1\otimes\beta_2\otimes\beta_3=t\beta_1\otimes t\beta_2\otimes t\beta_3)}.$$
Consider integral chains $C_1$, $C_2$, $C_3$, in $\tilde{X}$ such that $\sum_{1\leq j\leq 3} \codim(C_j)=3$. 
Assume $C_1$, $C_2$, $C_3$ are {\em $\tau$-transverse} in $\tilde{X}$, {\em i.e.} $\tau^{k_1}C_1$, 
$\tau^{k_2}C_2$, and $\tau^{k_3}C_3$ are transverse for all integers $k_1$, $k_2$, $k_3$. Define the {\em equivariant triple intersection number} 
$<C_1,C_2,C_3>_e$ by:
 $$<C_1,C_2,C_3>_e=\sum_{k_2\in\Z}\sum_{k_3\in\Z} <C_1,\tau^{-k_2}C_2,\tau^{-k_3}C_3>t_2^{k_2}t_3^{k_3} 
   \quad \in \frac{\Rd}{(t_1t_2t_3-1)},$$
where $\Rd=\Q[t_1^{\pm 1},t_2^{\pm 1},t_3^{\pm 1}]$.
Extend it to rational chains by multilinearity. 
We have the following easy formulae.
\begin{lemma} \label{lemmaformulae}
 The equivariant triple intersection number satisfies:
\begin{itemize}
 \item if $\codim(C_j)=1$ for all $j$, then for any permutation $\sigma\in \mathcal{S}_3$, with signature $\sg$, 
  $<C_{\sigma(1)},C_{\sigma(2)},C_{\sigma(3)}>_e(t_1,t_2,t_3)=\sg <C_1,C_2,C_3>_e(t_{\sigma^{-1}(1)},t_{\sigma^{-1}(2)},t_{\sigma^{-1}(3)})$, 
 \item $<P_1(\tau)C_1,P_2(\tau)C_2,P_3(\tau)C_3>_e=P_1(t_1)P_2(t_2)P_3(t_3)<C_1,C_2,C_3>_e$, for all $P_j\in\Qt$.
\end{itemize}
\end{lemma}

In Section \ref{seceqint}, we prove:
\begin{lemma} \label{lemmaindsurf}
 Let $(M,K)$ be a $\Q$SK-pair. Let $\tilde{X}$ 
be the infinite cyclic covering associated with $(M,K)$. Let $\beta_1$, $\beta_2$, $\beta_3$ be elements of $\Al(M,K)$ which can be represented 
by knots in $\tilde{X}$. Let $\mu_1$, $\mu_2$, $\mu_3$ be representatives of the $\beta_j$ 
whose images in $M$ are pairwise disjoint. For $j=1,2,3$, let $P_j\in\Qt$ satisfy $P_j(\tau)\mu_j=0$ in $\Al(M,K)$. 
Let $\Sigma_1$, $\Sigma_2$, $\Sigma_3$ be $\tau$-transverse rational 2-chains such that $\partial \Sigma_j=P_j(\tau)\mu_j$. 
Then $$<\Sigma_1,\Sigma_2,\Sigma_3>_e\,\in\frac{\Qti}{(t_1t_2t_3-1,P_1(t_1),P_2(t_2),P_3(t_3))}$$ does not depend 
on the choice of the surfaces $\Sigma_j$, and of the representatives $\mu_j$. 
\end{lemma}

Let $(M,K)$ be a $\Q$SK-pair. 
Set $\displaystyle \Rd_\delta=\frac{\Qti}{(t_1t_2t_3-1,\delta(t_1),\delta(t_2),\delta(t_3))}$, where $\delta(t)=\delta_{(M,K)}(t)$ 
is the annihilator of $\Al(M,K)$. Define a structure of $\Rd_\delta$-module on $\Ah(M,K)$ by:
$$t_1^{k_1}t_2^{k_2}t_3^{k_3}.\beta_1\otimes\beta_2\otimes\beta_3=t^{k_1}\beta_1\otimes t^{k_2}\beta_2\otimes t^{k_3}\beta_3.$$

Lemmas \ref{lemmaformulae} and \ref{lemmaindsurf} imply:
\begin{theorem} \label{thinttriple}
 Let $(M,K)$ be a $\Q$SK-pair. Let $\delta(t)=\delta_{(M,K)}(t)$. Let $\tilde{X}$ 
be the infinite cyclic covering associated with $(M,K)$. Define a $\Q$-linear map $\phi^{(M,K)} : \Ah(M,K) \to \Rd_\delta$ as follows. 
If $\mu_1$, $\mu_2$, $\mu_3$ are knots in $\tilde{X}$ whose images in $M\setminus K$ are pairwise disjoint, 
let $\Sigma_1$, $\Sigma_2$, $\Sigma_3$ be $\tau$-transverse rational 2-chains such that $\partial \Sigma_j=\delta(\tau)\mu_j$, 
and set $$\phi^{(M,K)}([\mu_1]\otimes[\mu_2]\otimes[\mu_3])=<\Sigma_1,\Sigma_2,\Sigma_3>_e.$$
Then the map $\phi^{(M,K)}$ is well-defined, $\Rd_\delta$-linear, and satisfies: 
\begin{equation} \tag{$\star$} \label{relperm} \phi^{(M,K)}(\otimes_{1\leq j\leq 3} \beta_{\sigma(j)})(t_1,t_2,t_3)=
\sg\phi^{(M,K)}(\otimes_{1\leq j\leq 3} \beta_j)(t_{\sigma^{-1}(1)},t_{\sigma^{-1}(2)},t_{\sigma^{-1}(3)}),\end{equation}
for all permutation $\sigma\in\perm$, with signature $\sg$, and all $(\beta_1,\beta_2,\beta_3)\in\Al(M,K)^3$. 
The isomorphism class of $(\Al(M,K),\phi^{(M,K)})$ is an invariant of the homeomorphism class of $(M,K)$. 
\end{theorem}

\paragraph{Remark}
So far, we do not need the condition that $K$ is null-homologous. Indeed, we do not even need to work in the exterior of a knot. 
Given an oriented 3-manifold equipped with a canonical infinite cyclic covering $\tilde{X}$, one can make the same construction 
on the torsion submodule of $H_1(\tilde{X};\Q)$, provided that $H_2(\tilde{X};\Q)=0$ (necessary in the proof of Lemma \ref{lemmaindsurf}). 
In this case, the variation under null borromean surgeries can also be computed as in Section \ref{secvar}. 

  \subsection{Variation under null borromean surgeries}

In order to define marked $\Q$SK-pairs, we recall the definition of the Blanchfield form. 

On an Alexander module $\Al(M,K)$, one can define the \emph{Blanchfield form}, or \emph{equivariant linking pairing},  
$\bl : \Al(M,K)\times\Al(M,K) \to \frac{\Q(t)}{\Qt}$, as follows. First define the equivariant linking number of two knots.
\begin{definition}
 Let $(M,K)$ be a $\Q$SK-pair. Let $\tilde{X}$ be the associated infinite cyclic covering. 
Let $\mu_1$ and $\mu_2$ be two knots in $\tilde{X}$ such that $\mu_1\cap \tau^k(\mu_2)=\emptyset$ for all $k\in\mathbb{Z}$.
Let $P\in\Qt$ satisfy $P(\tau)\mu_1=\partial S$, where S is an integral 2-chain in $\tilde{X}$. 
The \emph{equivariant linking number} of $\mu_1$ and $\mu_2$ is 
 $$lk_e(\mu_1,\mu_2)=\frac{1}{P(t)}\sum_{k\in\Z}<S,\tau^k(\mu_2)>t^k\quad \in\Q(t).$$
\end{definition}
One can easily see that $lk_e(\mu_1,\mu_2)\in\frac{1}{\delta(t)}\Qt$, $lk_e(\mu_2,\mu_1)(t)=lk_e(\mu_1,\mu_2)(t^{-1})$, and 
$lk_e(P(\tau)\mu_1,Q(\tau)\mu_2)(t)=P(t)Q(t^{-1})lk_e(\mu_1,\mu_2)(t)$.
Now, if $\beta_1$ (resp. $\beta_2$) is the homology class of $\mu_1$ (resp. $\mu_2$) in $\Al(M,K)$, define $\bl(\beta_1,\beta_2)$ by:
$$\bl(\beta_1,\beta_2)=lk_e(\mu_1,\mu_2)\ mod\ \Qt.$$
The Blanchfield form is {\em hermitian}: 
$$\bl(\beta_1,\beta_2)(t)=\bl(\beta_2,\beta_1)(t^{-1})\quad\textrm{and} \quad \bl(P(t)\beta_1,Q(t)\beta_2)(t)=P(t)Q(t^{-1})\,\bl(\beta_1,\beta_2)(t)$$ 
for all $\beta_1,\beta_2\in\Al(M,K)$ and all $P,Q\in\Qt$. 
Moreover, as proved by Blanchfield \cite{Bla}, it is {\em non degenerate}: $\bl(\beta_1,\beta_2)=0$ for all $\beta_2\in\Al(M,K)$ implies $\beta_1=0$.
The {\em Blanchfield module} of a $\Q$SK-pair is the Alexander module of the pair endowed with its Blanchfield form. 

Fix an abstract Blanchfield module $(\Al,\bl)$ (see \cite{M1} for a characterization of these modules). 
If $\xi$ is a fixed isomorphism from $(\Al,\bl)$ to the Blanchfield module of a $\Q$SK-pair $(M,K)$, 
then $(M,K,\xi)$ is an {\em $(\Al,\bl)$-marked $\Q$SK-pair}. Let $\Ens^m(\Al,\bl)$ be the set of all such $(\Al,\bl)$-marked $\Q$SK-pairs 
up to orientation-preserving and marking-preserving homeomorphism. 
When it does not seem to cause confusion, the image of an element $\beta\in\Al$ by a marking $\xi$ is still denoted by $\beta$, 
and an $(\Al,\bl)$-marked $\Q$SK-pair is called a marked $\Q$SK-pair. Note that the infinite cyclic covering $\tilde{X}$ 
associated with a $\Q$SK-pair $(M,K)$ is well-defined only up to the automorphisms of the covering, which are the $\tau^k$. 
Hence a marking $\xi$ of $(M,K)$ is defined up to multiplication by a power of $t$. 

For a marked $\Q$SK-pair $(M,K,\xi)$, the equivariant triple intersection map defined in Theorem \ref{thinttriple} is well-defined 
on $\displaystyle\Ah=\frac{\Al\otimes\Al\otimes\Al}{(\otimes_{1\leq j\leq 3} \beta_j=\otimes_{1\leq j\leq 3} t\beta_j)}$, 
not only up to isomorphism, and we denote it by $\phi^{(M,K,\xi)}$. 
We aim at studying the variation of the map $\phi^{(M,K,\xi)}$ under null borromean surgeries, that we now define.

\begin{figure}[htb] 
\begin{center}
\begin{tikzpicture} [scale=0.15]
\newcommand{\feuille}[1]{
\draw[rotate=#1,thick,color=gray] (0,-11) circle (5);
\draw[rotate=#1,thick,color=gray] (0,-11) circle (1);
\draw[rotate=#1,line width=8pt,color=white] (-2,-6.42) -- (2,-6.42);
\draw[rotate=#1,thick,color=gray] (2,-1.15) -- (2,-6.42);
\draw[rotate=#1,thick,color=gray] (-2,-1.15) -- (-2,-6.42);
\draw[white,rotate=#1,line width=5pt] (0,0) -- (0,-8);
\draw[rotate=#1] (0,0) -- (0,-8);
\draw[rotate=#1] (0,-11) circle (3);
\draw[->,rotate=#1] (-3,-10.9) -- (-3,-11.1);}
\draw (0,0) circle (1.5);
\draw[->] (0.1,1.5) -- (-0.1,1.5);
\feuille{0}
\feuille{120}
\feuille{-120}
\draw (-4,10) node{$\scriptstyle{\textrm{leaf}}$};
\draw[->] (-5,9) -- (-6.3,7.5);
\draw (11.5,-1) node{$\scriptstyle{\textrm{internal vertex}}$};
\draw[<-] (0.5,-0.1) -- (4,-1);
\draw (4.7,-9.2) node{$\Gamma_0$};
\draw[color=gray] (6.5,-16.8) node{$\Sigma(\Gamma_0)$};
\end{tikzpicture}
\end{center}
\caption{The standard Y-graph}\label{figY}
\end{figure}
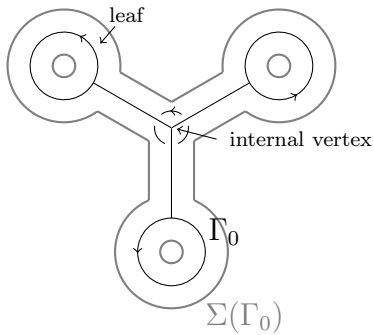
The {\em standard Y-graph} is the graph $\Gamma_0\subset \mathbb{R}^2$ represented in Figure \ref{figY}. 
The looped edges of $\Gamma_0$ are the \emph{leaves}. 
The vertex incident to three different edges is the {\em internal vertex}. 
With $\Gamma_0$ is associated a regular neighborhood $\Sigma(\Gamma_0)$ of $\Gamma_0$ in the plane. 
The surface $\Sigma(\Gamma_0)$ is oriented with the usual convention. This induces an orientation of the leaves, 
and an {\em orientation of the internal vertex}, {\em i.e.} a cyclic order of the three edges which meet at this vertex. 
Let $M$ be a 3-manifold and let $h:\Sigma(\Gamma_0)\to M$ be an embedding. The image $\Gamma$ of $\Gamma_0$ is a {\em Y-graph}, endowed with  
its {\em associated surface} $\Sigma(\Gamma)=h(\Sigma(\Gamma_0))$. The Y-graph $\Gamma$ is equipped with 
the framing induced by $\Sigma(\Gamma)$. 

\begin{figure}[hbt] 
\begin{center}
\begin{tikzpicture} [scale=0.15]
\begin{scope}
\newcommand{\feuille}[1]{
\draw[rotate=#1] (0,0) -- (0,-8);
\draw[rotate=#1] (0,-11) circle (3);}
\feuille{0}
\feuille{120}
\feuille{-120}
\draw (3,-4) node{$\Gamma$};
\end{scope}
\draw[very thick,->] (21.5,-3) -- (23.5,-3);
\begin{scope}[xshift=1200]
\newcommand{\bras}[1]{
\draw[rotate=#1] (0,-1.5) circle (2.5);
\draw [rotate=#1,white,line width=8pt] (-0.95,-4) -- (0.95,-4);
\draw[rotate=#1] {(0,-11) circle (3) (1,-3.9) -- (1,-7.6)};
\draw[rotate=#1,white,line width=6pt] (-1,-5) -- (-1,-8.7);
\draw[rotate=#1] {(-1,-3.9) -- (-1,-8.7) (-1,-8.7) arc (-180:0:1)};}
\bras{0}
\draw [white,line width=6pt,rotate=120] (0,-1.5) circle (2.5);
\bras{120}
\draw [rotate=-120,white,line width=6pt] (-1.77,0.27) arc (135:190:2.5);
\draw [rotate=-120,white,line width=6pt] (1.77,0.27) arc (45:90:2.5);
\bras{-120}
\draw [white,line width=6pt] (-1.77,0.27) arc (135:190:2.5);
\draw [white,line width=6pt] (1.77,0.27) arc (45:90:2.5);
\draw (-1.77,0.27) arc (135:190:2.5);
\draw (1.77,0.27) arc (45:90:2.5);
\draw (3.5,-4.5) node{$L$};
\end{scope}
\end{tikzpicture}
\end{center}
\caption{Y-graph and associated surgery link}\label{figborro}
\end{figure}
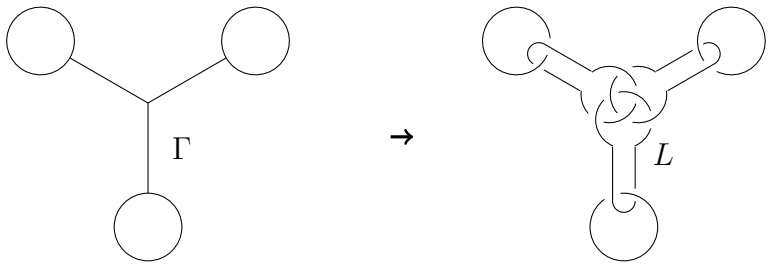

Let $\Gamma$ be a Y-graph in a 3-manifold $M$. Let $\Sigma(\Gamma)$ be its associated surface. In $\Sigma(\Gamma)\times[-1,1]$, 
associate with $\Gamma$ the six-component link $L$ represented in Figure \ref{figborro}, with the blackboard framing. 
The \emph{borromean surgery on $\Gamma$} is the usual surgery along the framed link $L$. The manifold obtained from $M$ 
by surgery on $\Gamma$ is denoted by $M(\Gamma)$. 

Let $(M,K,\xi)\in\Ens^m(\Al,\bl)$. Let $\Gamma$ be a Y-graph in $M\setminus K$. 
If the map $i_*:H_1(\Gamma;\Q)\to H_1(M\setminus K)$ induced by the inclusion has a trivial image, then $\Gamma$ is {\em null in 
$M\setminus K$}, and the surgery on $\Gamma$ is a {\em null borromean surgery} (null-move in \cite{GR}). 
In this case, the pair $(M,K)(\Gamma)$ obtained from $(M,K)$ by surgery on $\Gamma$ is again a $\Q$SK-pair. 
The surgery on $\Gamma$ induces a canonical isomorphism between the Blanchfield modules of $(M,K)$ and $(M,K)(\Gamma)$ (see \cite[Lemma 2.1]{M3}).
Hence we can define the marked $\Q$SK-pair $(M,K,\xi)(\Gamma)$ obtained from $(M,K,\xi)$ by surgery on $\Gamma$.

In Section \ref{secvar}, we prove:
\begin{proposition} \label{propdiff}
Let $(M,K,\xi)\in\Ens^m(\Al,\bl)$. Let $\Gamma$ be a Y-graph, null in $M\setminus K$. Let $\tilde{\Gamma}$ be a lift of $\Gamma$ 
in the infinite cyclic covering $\tilde{X}$ associated with $(M,K)$. Let $\gamma_1$, $\gamma_2$, $\gamma_3$ be the leaves of $\tilde{\Gamma}$ 
in $\Al$, given in an order induced by the orientation of the internal vertex of $\Gamma$. 
For $\beta_1$, $\beta_2$, $\beta_3$ in $\Al$:
$$\phi^{(M,K,\xi)(\Gamma)}(\beta_1\otimes\beta_2\otimes\beta_3)-\phi^{(M,K,\xi)}(\beta_1\otimes\beta_2\otimes\beta_3)=
\sum_{\sigma\in\perm}\sg\prod_{j=1}^3 \delta(t_j) \bl(\beta_j,[\gamma_{\sigma(j)}])(t_j).$$
\end{proposition}

The following corollary says that the triple intersection map is a degree one invariant of $(\Al,\bl)$-marked $\Q$SK-pairs 
with respect to null borromean surgeries.
\begin{corollary} \label{corinv}
 Let $(M,K,\xi)\in\Ens^m(\Al,\bl)$. Let $\Gamma_1$ and $\Gamma_2$ be disjoint Y-graphs, null in $M\setminus K$. 
Then the map $\phi^{(M,K,\xi)}-\phi^{(M,K,\xi)(\Gamma_1)}-\phi^{(M,K,\xi)(\Gamma_2)}+\phi^{(M,K,\xi)(\Gamma_1)(\Gamma_2)}$ 
vanishes on $\Ah$.
\end{corollary}
\begin{proof}
 Since the Blanchfield form is preserved by null borromean surgeries, it follows from Proposition \ref{propdiff} 
that the difference $\phi^{(M,K,\xi)}-\phi^{(M,K,\xi)(\Gamma_1)}$ is not modified when performing the surgery on $\Gamma_2$.
\end{proof}

Proposition \ref{propdiff} will allow us to give a description of the space of all equivariant triple intersection maps. More precisely, 
let $\maps$ be the rational vector space of all morphisms of $\Rd_\delta$-modules $\phi:\Ah\to \Rd_\delta$ which satisfy the relation 
(\ref{relperm}) of Theorem \ref{thinttriple}. 
In Section \ref{seccar}, we prove:
\begin{theorem} \label{thcar} 
Define $\phi^{\bullet}: \Ens^m(\Al,\bl)\to\maps$ by $\phi^{\bullet}(M,K,\xi)=\phi^{(M,K,\xi)}$. 
Then the rational vector space $\phi^{\bullet}(\Ens^m(\Al,\bl))$ is isomorphic to 
$\displaystyle \HA=\frac{\Lambda^3\Al}{(\beta_1\wedge\beta_2\wedge\beta_3=t\beta_1\wedge t\beta_2\wedge t\beta_3)}$.
\end{theorem}

   \subsection{Structure of $\HA$} \label{subsecH}

Fix an abstract Blanchfield module $(\Al,\bl)$. In Section \ref{secH}, we study the structure of 
$$\Al_h=\frac{\Al^{\otimes 3}}{(\beta_1\otimes\beta_2\otimes\beta_3=t\beta_1\otimes t\beta_2\otimes t\beta_3)}$$
and 
$$\HA=\frac{\Lambda^3\Al}{(\beta_1\wedge\beta_2\wedge\beta_3=t\beta_1\wedge t\beta_2\wedge t\beta_3)}.$$
For this study, we consider a decomposition of $\Al$ as a direct sum of cyclic submodules, and associated decompositions of $\Ah$ and $\HA$. 
In order to characterize the equivariant triple intersection maps in Section \ref{seccar}, we choose a decomposition adapted 
to the Blanchfield form. 

By \cite[Theorem 1.3]{M1}, the $\Qt$-module $\Al$ is a direct sum, orthogonal with respect to $\bl$, of submodules of these two kinds 
($\pi\in\Qt$ is {\em symmetric} if $\pi(t^{-1})=rt^k\pi(t)$ with $r\in\Q^*$ and $k\in\Z$):
\begin{itemize}
 \item $\frac{\Qt}{(\pi^n(t))}\eta$, with $\pi$ prime and symmetric, or $\pi(t)=t+2+t^{-1}$, $n>0$,
  and $\bl(\eta,\eta)=\frac{a}{\pi^n}$, $a$ symmetric and prime to $\pi$,
 \item $\frac{\Qt}{(\pi^{n}(t))}\eta \oplus \frac{\Qt}{(\pi^{n}(t^{-1}))}\eta'$,
  with either $\pi$ prime, non symmetric, $\pi(-1)\neq 0$, $n>0$, or $\pi(t)=1+t$, $n$ odd, and in both cases 
 $\bl(\eta,\eta')=\frac{1}{\pi^n}$, $\bl(\eta,\eta)=\bl(\eta',\eta')=0$.
\end{itemize}
Define ``Blanchfield duals'' for the generators:
\begin{itemize}
 \item in the first case, set $d(\eta)=\eta$,
 \item in the second case, set $d(\eta)=\eta'$, and $d(\eta')=\eta$.
\end{itemize}
Index all these generators to obtain a family $(\eta_i)_{1\leq i\leq q}$ that generates $\Al$ over $\Qt$.
We finally have a family $(\eta_i)_{1\leq i\leq q}$ in $\Al$, an involution $d$ of that family, and polynomials $a_i$, $\delta_i$ in $\Qt$, 
that satisfy:
\begin{itemize}
 \item $\Al=\bigoplus_{i=1}^q \Al_i$, where $\Al_i= \frac{\Qt}{(\delta_i)} \eta_i$,
 \item $\bl(\eta_i,d(\eta_j))=0$ if $i\neq j$,
 \item each $\delta_i$ is a power of a prime polynomial, 
 \item $\bl(\eta_i,d(\eta_i))=\frac{a_i}{\delta_i}$, where $a_i$ is prime to $\delta_i$.
\end{itemize}
For technical simplicity, we denote by $m_i$ the power that appears when we write $\delta_i$ as a power of a prime polynomial, and we ask that 
$m_i\geq m_{i+1}$ for $1\leq i<q$. Note that $m_i$ is the multiplicity of any complex root of $\delta_i$. 
Normalize the $\delta_i$ so that $\delta_i(t)\in\Q[t]$, $\delta_i(0)\neq0$ and $\delta_i(1)=1$.

The well-known result on the structure of finitely generated modules over a principal ideal domain implies that 
the family of the $\delta_i$'s is well-defined up to permutation. Hence if $\Al=\oplus_{1\leq i\leq q'} \Al'_i$ 
is another decomposition of $\Al$ satisfaying the above conditions, then $q'=q$ and there is a permutation $\sigma$ 
of $\{1,..,q\}$ such that $\Al'_i$ is isomorphic to $\Al_{\sigma(i)}$. But the decomposition $\Al=\bigoplus_{i=1}^q \Al_i$ is not unique. 
For instance, if $\Al=\frac{\Qt}{(\delta)}\eta_1\oplus\frac{\Qt}{(\delta)}\eta_2$ with $\bl(\eta_1,\eta_1)=\bl(\eta_2,\eta_2)$ 
and $\bl(\eta_1,\eta_2)=0$, then the decomposition $\Al=\frac{\Qt}{(\delta)}(\eta_1+\eta_2)\oplus\frac{\Qt}{(\delta)}(\eta_1-\eta_2)$ 
also satisfies the above conditions. When the $\Al_i$'s are fixed, it remains infinitely many possible choices for the generators $\eta_i$. 

For $\mi=(i_1,i_2,i_3)\in\{1,..,q\}^3$, set:
$$\Al(\mi)=\frac{\Al_{i_1}\otimes\Al_{i_2}\otimes\Al_{i_3}}{(\otimes_{1\leq j\leq 3} \beta_j =\otimes_{1\leq j\leq 3} t\beta_j)}.$$
We have: 
$$\Ah=\bigoplus_{\mi\in\{1,..,q\}^3} \Al(\mi).$$

For $\mi=(i_1,i_2,i_3)\in\{1,..,q\}^3$, let $\HA(\mi)$ be the rational vector subspace of $\HA$ generated by the 
$t^{k_1}\eta_{i_1}\wedge t^{k_2}\eta_{i_2}\wedge t^{k_3}\eta_{i_3}$ for all integers $k_1$, $k_2$, $k_3$. We have:
$$\HA=\bigoplus_{1\leq i_1\leq i_2\leq i_3\leq q} \HA(\mi).$$

In Section \ref{secH}, we prove the following results, and we further study the structure of the $\Al(\mi)$ and $\HA(\mi)$ 
in order to bound their dimensions. 
\begin{theorem} \label{thstructureAh}
 Let $\mi=(i_1,i_2,i_3)\in\{1,..,q\}^3$. 
The rational vector space $\Al(\mi)$ is non trivial if and only if there are complex roots $z_1$, $z_2$, $z_3$, of $\delta_{i_1}$, 
$\delta_{i_2}$, $\delta_{i_3}$, respectively, such that $z_1z_2z_3=1$.
\end{theorem}

\begin{theorem} \label{thstructureh}
 Let $\mi=(i_1,i_2,i_3)\in\{1,..,q\}^3$. 
The rational vector space $\HA(\mi)$ is non trivial if and only if there are complex roots $z_1$, $z_2$, $z_3$, 
of $\delta_{i_1}$, $\delta_{i_2}$, $\delta_{i_3}$ respectively, which satisfy:
\begin{itemize}
 \item $z_1z_2z_3=1$, 
 \item for $1\leq j\leq 3$, the multiplicity $m_{i_j}$ is at least the number of indices $l\in\{1,2,3\}$ such that 
  $i_l=i_j$ and $z_l=z_j$.
\end{itemize}
\end{theorem}

\paragraph{Example}
If all the roots of the Alexander polynomial are simple, and if the product of three of them is always different from 1, then 
$\HA=0$. It is the case, for instance, of the trefoil knot, and of the figure eight knot, in $S^3$. We will study non trivial examples 
in Section \ref{secH}.

\subsection{Degree one invariants of marked $\Q$SK-pairs} \label{subsecgrad}

In this subsection, we describe the finiteness and universality properties of the equivariant triple intersection map. 
Let us define Lagrangian-preserving surgeries. 
\begin{definition}
 For $g\in \N$, a \emph{genus $g$ rational homology handlebody} ($\Q$HH) 
is a 3-manifold which is compact, oriented, and which has the same homology with rational coefficients 
as the standard genus $g$ handlebody.
\end{definition}
Such a $\Q$HH is connected, and its boundary is necessarily homeomorphic to the standard genus $g$ surface.

\begin{definition} \label{deflag}
The \emph{Lagrangian} $\mathcal{L}_A$ of a $\Q$HH $A$ is the kernel of the map 
$$i_*: H_1(\partial A;\Q)\to H_1(A;\Q)$$
induced by the inclusion. Two $\Q$HH's $A$ and $B$ have \emph{LP-identified} boundaries if $(A,B)$ is equipped with a homeomorphism 
$h:\partial A\to\partial B$ such that $h_*(\mathcal{L}_A)=\mathcal{L}_B$.
\end{definition}
The Lagrangian of a $\Q$HH $A$ is indeed a Lagrangian subspace of $H_1(\partial A;\Q)$ with respect to the intersection form.

Let $M$ be a $\Q$HS, let $A\subset M$ be a $\Q$HH, and let $B$ be a $\Q$HH whose boundary is LP-identified 
with $\partial A$. Set $M(\frac{B}{A})=(M\setminus Int(A))\cup_{\partial A=_h\partial B}B$. We say that the $\Q$HS 
$M(\frac{B}{A})$ is obtained from $M$ by \emph{Lagrangian-preserving surgery}, or \emph{LP-surgery}. 
 
Given a $\Q$SK-pair $(M,K)$, a \emph{$\Q$HH null in $M\setminus K$} is a $\Q$HH $A\subset M\setminus K$ such that 
the map $i_* : H_1(A;\Q)\to H_1(M\setminus K;\Q)$ induced by the inclusion has a trivial image.
A \emph{null LP-surgery} on $(M,K)$ is an LP-surgery $(\frac{B}{A})$ such that $A$ is null in $M\setminus K$. 
The $\Q$SK-pair obtained by surgery is denoted by $(M,K)(\frac{B}{A})$. 
Since a null LP-surgery induces a canonical isomorphism beetween the Blanchfield modules of the involved pairs (see Theorem \ref{thM3} below), 
this move is well-defined on marked $\Q$SK-pairs. The marked $\Q$SK-pair obtained from a marked $\Q$SK-pair $(M,K,\xi)$ 
by a null LP-surgery $(\frac{B}{A})$ is denoted by $(M,K,\xi)(\frac{B}{A})$. 

A borromean surgery along a Y-graph $\Gamma$ in a $3$-manifold $N$ can be realized by cutting a regular neighborhood of $\Gamma$ in $N$ 
(a genus 3 standard handlebody), and gluing another genus 3 handlebody instead, in a Lagrangian-preserving way (see \cite{Mat}). 
Hence borromean surgeries are a specific kind of LP-surgeries. 

Let $\F^m_0$ be the rational vector space generated by all marked $\Q$SK-pairs up to orien\-ta\-tion-preserving homeomorphism. 
Let $\F^m_n$ denote the subspace of $\F^m_0$ generated by the 
$$[(M,K,\xi);(\frac{B_i}{A_i})_{1\leq i \leq n}]=\sum_{I\subset \{ 1,...,n\}} (-1)^{|I|} (M,K,\xi)((\frac{B_i}{A_i})_{i\in I})$$ 
for all marked $\Q$SK-pairs $(M,K,\xi)$ and all families of $\Q$HH's $(A_i,B_i)_{1\leq i \leq n}$, where the $A_i$ are null in $M\setminus K$ 
and disjoint, and each $\partial B_i$ is 
LP-identified with the corresponding $\partial A_i$. Since $\F^m_{n+1}\subset \F^m_n$, this defines a filtration. 
\begin{theorem}[\cite{M3} Theorem 1.13] \label{thM3}
 A null LP-surgery induces a canonical isomorphism between the Blanchfield modules of the involved $\Q$SK-pairs. 
Conversely, any isomorphism between the Blanchfield modules of two $\Q$SK-pairs 
can be realized by a finite sequence of null LP-surgeries, up to multiplication by a power of $t$.
\end{theorem}

This result implies in particular that the fitration $(\F^m_n)_{n\in\N}$ splits in the following 
way. For a given Blanchfield module $(\Al,\bl)$, let $\F^m_0(\Al,\bl)$ be the subspace of $\F^m_0$ generated 
by the $(\Al,\bl)$-marked $\Q$SK-pairs. Let $(\F^m_n(\Al,\bl))_{n\in\N}$ be the filtration defined on $\F^m_0(\Al,\bl)$ 
by null LP-surgeries. Then, for $n\in\N$, $\F^m_n$ is the direct sum over all isomorphism classes of Blanchfield modules of the 
$\F^m_n(\Al,\bl)$. Set $\G^m_n(\Al,\bl)=\F^m_n(\Al,\bl) / \F^m_{n+1}(\Al,\bl)$ 
and $\G^m(\Al,\bl)=\oplus_{n\in\N} \G^m_n(\Al,\bl)$. 

An invariant of $(\Al,\bl)$-marked $\Q$SK-pairs is a map defined on $\Ens^m(\Al,\bl)$. Given such an invariant $\lambda$ valued 
in an abelian torsion free group $Z$, one can extend it into a $\Q$-linear map $\tilde{\lambda}:\F^m_0(\Al,\bl)\to\Q\otimes_\Z Z$. 
The invariant $\lambda$ is a 
\emph{finite type invariant of degree at most $n$ of $(\Al,\bl)$-marked $\Q$SK-pairs with respect to null LP-surgeries} 
if $\tilde{\lambda}(\F^m_{n+1}(\Al,\bl))=0$. 
The dual of the quotient $\G^m_n(\Al,\bl)$ is naturally identified with the 
space of all rational valued finite type invariants of degree $n$ of marked $\Q$SK-pairs with respect to null LP-surgeries, 
hence a description of $\G^m_n(\Al,\bl)$ provides a description of this space of invariants. Theorem~\ref{thM3} implies $\G^m_0(\Al,\bl)\cong\Q$. 

We studied in \cite[Chap. 6]{Mt} the filtration associated to $\Q$SK-pairs (without marking) and defined a graded space 
of diagrams which surjects onto the corresponding graded space $\G(\Al,\bl)$. This work can be adapted to marked $\Q$SK-pairs in order 
to define a graded space of diagrams and a surjective map from this space to $\G^m(\Al,\bl)$. 
We focuse here on the degree one case, and we give a complete description of $\G^m_1(\Al,\bl)$ 
for an arbitrary isomorphism class $(\Al,\bl)$ of Blanchfield modules. 

In Subsection \ref{subphi}, in order to prove Theorem \ref{thcar}, we construct an isomorphism 
$\h:\phi^{\bullet}(\Ens^m(\Al,\bl))\fl{\scriptstyle{\cong}}\HA$. 
Set $\hs=\h\circ\phi^{\bullet} : \Ens^m(\Al,\bl) \to \HA$. The following result 
is a consequence of Theorem \ref{thdecsur}, Corollary \ref{corinv} and Lemma \ref{lemmainv}.
\begin{proposition} \label{propdegree1}
 The map $\hs : \Ens^m(\Al,\bl) \to \HA$ is a degree at most one invariant of $(\Al,\bl)$-marked $\Q$SK-pairs with respect to null LP-surgeries. 
\end{proposition}

For a prime integer $p$, define a map $\nu_p : \F^m_0\to\Q$ by $\nu_p(M,K,\xi)=v_p(|H_1(M;\Z)|)$, where $v_p$ is the $p$-adic valuation, 
and $|.|$ denotes the cardinality. By \cite[Proposition~0.8]{M2}, the $\nu_p$ are degree 1 invariants of $\Q$HS's, hence they are 
also degree 1 invariants of $\Q$SK-pairs. The following result is obtained in Section \ref{secG1} as a consequence of Propositions~\ref{propborro} 
and~\ref{propG1}. 
\begin{theorem} \label{thG1}
Fix a Blanchfield module $(\Al,\bl)$. 
Set $\displaystyle \HA\hspace{-1.5pt}=\hspace{-1.5pt}\frac{\Lambda^3\Al}
{(\beta_1\wedge\beta_2\wedge\beta_3\hspace{-2pt}=\hspace{-2pt}t\beta_1\wedge t\beta_2\wedge t\beta_3)}$. 
Let $(M,K,\xi)\in\Ens^m(\Al,\bl)$. For $p$ prime, let $B_p$ be a rational homology ball such that $H_1(B_p;\Z)=\Z/p\Z$. Then:
$$\G^m_1(\Al,\bl)\cong\left(\bigoplus_{p\textrm{ prime}} \Q [(M,K,\xi);\frac{B_p}{B^3}]\right)\bigoplus\HA.$$
\end{theorem}
Moreover, Propositions~\ref{propborro} and~\ref{propG1} show that the invariants $\nu_p$ together with the map $\hs$, obtained from the 
equivariant triple intersection map, form a universal rational valued finite type invariant of degree 1 of $(\Al,\bl)$-marked $\Q$SK-pairs 
with respect to null LP-surgeries, in the following sense. 
If $\lambda:\Ens^m(\Al,\bl)\to \Q$ is a degree 1 invariant with respect to null LP-surgeries, then there are maps $f:\HA\to\Q$ 
and $g_p:\Q\to\Q$ for all prime integer~$p$, such that $\lambda-(f\circ\hs+\sum_{p\textrm{ prime}}g_p\circ\nu_p)$ is a degree $0$ 
invariant, {\em i.e.} a constant. 

\paragraph{The case of $\Z$SK-pairs}
A {\em $\Z$SK-pair} $(M,K)$ is a $\Q$SK-pair such that $M$ is an {\em integral homology 3-sphere}, 
{\em i.e.} an oriented compact 3-manifold which has the same homology with integral coefficients as the standard 
3-sphere $S^3$. The {\em integral Alexander module} of a $\Z$SK-pair $(M,K)$ is the $\Zt$-module $\Al_\Z(M,K)=H_1(\tilde{X};\Z)$, where 
$\tilde{X}$ is the infinite cyclic covering associated with $(M,K)$. The {\em integral Blanchfield module} of $(M,K)$ 
is the integral Alexander module $\Al_\Z(M,K)$ equipped with the Blanchfield form. Fix an integral Blanchfield module $(\Al_\Z,\bl)$. 
If $\xi$ is a fixed isomorphism from $(\Al_\Z,\bl)$ to the Blanchfield module of a $\Z$SK-pair $(M,K)$, then $(M,K,\xi)$ 
is an {\em $(\Al_\Z,\bl)$-marked $\Z$SK-pair}. As for $\Q$SK-pairs, this isomorphism $\xi$ is defined up to multiplication by a power of $t$. 
Let $\Ens^m_\Z(\Al_\Z,\bl)$ be the set of all such $(\Al_\Z,\bl)$-marked $\Z$SK-pairs up to orientation-preserving and marking-preserving homeomorphism, 
called marked $\Z$SK-pairs when it does not seem to cause confusion. 

Borromean surgeries are well defined on the set of marked $\Z$SK-pairs, since they preserve the homology 
of the manifold. The equivariant triple intersection map is again a degree one invariant of marked $\Z$SK-pairs 
with respect to null borromean surgeries. We will see that this invariant contains all the rational valued degree one invariants 
of marked $\Z$SK-pairs with respect to null borromean surgeries. 

Replacing $\Q$ by $\Z$ in the definitions at the beginning of the subsection, define {\em integral homology handlebodies} ($\Z$HH), 
{\em integral Lagrangians}, {\em integral LP-surgeries}, and {\em integral null LP-surgeries}, similarly. 
Integral LP-surgeries (in particular borromean surgeries) preserve the homology with integral coefficients of the manifold. 
Hence they provide a move on the set of integral homology 3-spheres. Integral null LP-surgeries define a move on the set of 
$\Z$SK-pairs. Moreover, they induce canonical isomorphisms 
beetween the integral Blanchfield modules of the involved pairs (see Theorem \ref{thM3Z} below), hence they provide a move 
on the set of marked $\Z$SK-pairs.

Let $\F^{m,\Z}_0$ be the rational vector space generated by all marked $\Z$SK-pairs up to orientation-preserving homeomorphism. 
Let $(\F^{m,\Z}_n)_{n\in\N}$ be the filtration of $\F^{m,\Z}_0$ defined by integral null LP-surgeries. The following result implies 
that borromean surgeries define the same filtration.
\begin{proposition}[\cite{AL} Lemma 4.11] \label{propAL}
 Let $A$ and $B$ be $\Z$HH's whose boundaries are LP-identified. Then $A$ and $B$ can be obtained from one another by a finite 
sequence of borromean surgeries in the interior of the $\Z$HH's.
\end{proposition}

The following result is the equivalent of Theorem \ref{thM3} in the setting of $\Z$SK-pairs.
\begin{theorem}[\cite{M3} Theorem 1.14] \label{thM3Z}
 An integral null LP-surgery induces a canonical isomorphism between the integral Blanchfield modules of the involved $\Z$SK-pairs. 
Conversely, any isomorphism between the integral Blanchfield modules of two $\Z$SK-pairs 
can be realized by a finite sequence of integral null LP-surgeries, up to multiplication by a power of~$t$.
\end{theorem}
This implies that the filtration $(\F^{m,\Z}_n)_{n\in\N}$ splits along the isomorphism classes of integral Blanchfield modules. 
For a given integral Blanchfield module $(\Al_\Z,\bl)$, let $\F^{m,\Z}_0(\Al_\Z,\bl)$ be the subspace of $\F^{m,\Z}_0$ generated 
by the $(\Al_\Z,\bl)$-marked $\Z$SK-pairs. Let $(\F^{m,\Z}_n(\Al_\Z,\bl))_{n\in\N}$ be the filtration defined 
on $\F^{m,\Z}_0(\Al_\Z,\bl)$ by integral null LP-surgeries. Then, for $n\in\N$, $\F^{m,\Z}_n$ is the direct sum over all isomorphism 
classes of integral Blanchfield modules of the $\F^{m,\Z}_n(\Al_\Z,\bl)$. Set 
$\G^{m,\Z}_n(\Al_\Z,\bl)=\F^{m,\Z}_n(\Al_\Z,\bl) / \F^{m,\Z}_{n+1}(\Al_\Z,\bl)$. 
Theorem~\ref{thM3Z} implies $\G^{m,\Z}_0(\Al_\Z,\bl)\cong\Q$. 

An invariant of $(\Al_\Z,\bl)$-marked $\Z$SK-pairs is a map defined on $\Ens^{m,\Z}(\Al_\Z,\bl)$. Given such an invariant 
$\lambda$ valued in an abelian torsion free group $Z$, one can extend it into a $\Q$-linear map 
$\tilde{\lambda}:\F^{m,\Z}_0(\Al_\Z,\bl)\to\Q\otimes_\Z Z$. 
The invariant $\lambda$ is a \emph{finite type invariant of degree at most $n$ of $(\Al_\Z,\bl)$-marked $\Z$SK-pairs 
with respect to integral null LP-surgeries} if $\tilde{\lambda}(\F^m_{n+1}(\Al_\Z,\bl))=0$. 

Consider the map $\hs$ of Proposition \ref{propdegree1} and its restriction $\hs:\Ens^{m,\Z}(\Al_\Z,\bl)\to\HA$. Corollary \ref{corinv} 
implies:
\begin{proposition}
 The map $\hs:\Ens^{m,\Z}(\Al_\Z,\bl)\to\HA$ is a degree at most one invariant of $(\Al_\Z,\bl)$-marked $\Z$SK-pairs 
with respect to integral null LP-surgeries.
\end{proposition}

In Section \ref{secG1}, we prove:
\begin{theorem} \label{thG1Z}
Fix an integral Blanchfield module $(\Al_\Z,\bl)$. Set $\Al=\Al_\Z\otimes_\Z\Q$. 
Set $\displaystyle \HA=\frac{\Lambda^3_\Q\Al}{(\beta_1\wedge\beta_2\wedge\beta_3=t\beta_1\wedge t\beta_2\wedge t\beta_3)}$. 
Then the map $\hs:\Ens^{m,\Z}(\Al_\Z,\bl)\to\HA$ induces an isomorphism: $$\G^{m,\Z}_1(\Al_\Z,\bl)\cong\HA.$$
\end{theorem}
This result shows that the map $\hs$, obtained from the equivariant triple intersection map, is a universal rational valued 
finite type invariant of degree 1 of $(\Al_\Z,\bl)$-marked $\Z$SK-pairs with respect to integral null LP-surgeries, 
in the following sense. 
If $\lambda:\Ens^{m,\Z}(\Al_\Z,\bl)\to \Q$ is a degree 1 invariant with respect to integral null LP-surgeries, then there is a map $f:\HA\to\Q$ 
such that $\lambda-f\circ\hs$ is a degree $0$ invariant, {\em i.e.} a constant.

    \section{Equivariant triple intersections} \label{seceqint}

In this section, we prove Lemma \ref{lemmaindsurf}. 

\begin{lemma} \label{subh2}
Let $(M,K)$ be a $\Q$SK-pair. Let $\tilde{X}$ be the associated infinite cyclic covering. Then $H_2(\tilde{X};\Q)=0$.
\end{lemma}
\begin{proof}
Let $\Sigma$ be a compact connected oriented surface embedded in $M$ such that $\partial \Sigma=K$. Set $V=M\setminus(\Sigma\times[-1,1])$. 
Note that $V$ is a rational homology handlebody (see \cite[Lemma 3.1]{M3}). In particular, $H_2(V;\Q)=0$. 
The boundary of $V$ is the union of $\Sigma^+=\Sigma\times\{1\}$, $\Sigma^-=\Sigma\times\{-1\}$, and $\partial\Sigma\times[-1,1]$. 
Consider $\Z$ copies $V_i$ of $V$, and let $\Sigma_i^+$, $\Sigma_i^-$ be the copies of $\Sigma^+$ and $\Sigma^-$ in $V_i$. 
The covering $\tilde{X}$ can be constructed by connecting all the $V_i$, gluing $\Sigma_i^-$ and $\Sigma_{i+1}^+$ for all $i\in\Z$. 
Set $\tilde{V}_e=\cup_{i\in\Z}V_{2i}$ and $\tilde{V}_o=\cup_{i\in\Z} V_{2i+1}$. Let $\tilde{\Sigma}$ be the preimage 
of $\Sigma$ in $\tilde{X}$, made of $\Z$ disjoint copies of $\Sigma$. We have $\tilde{\Sigma}=\tilde{V}_e\cap \tilde{V}_o$. 
The Mayer-Vietoris sequence associated with $\tilde{X}=\tilde{V}_e\cup \tilde{V}_o$ yields the exact sequence:
$$H_2(\tilde{V}_e;\Q)\oplus H_2(\tilde{V}_o;\Q)\longrightarrow H_2(\tilde{X};\Q)\longrightarrow H_1(\tilde{\Sigma};\Q)
\fl{\iota} H_1(\tilde{V}_e;\Q)\oplus H_1(\tilde{V}_o;\Q).$$
The module $H_2(\tilde{V}_e;\Q)\oplus H_2(\tilde{V}_o;\Q)$ is a direct sum of $\Z$ copies of $H_2(V;\Q)$, which is trivial. Hence 
$H_2(\tilde{V}_e;\Q)\oplus H_2(\tilde{V}_o;\Q)=0$. It is well-known that the map $\iota$ provides a square, non degenerate presentation of 
the Alexander module (see \cite[Theorem 6.5]{Lick} for details). In particular, $\iota$ is known to be injective. Finally $H_2(\tilde{X};\Q)=0$.
\end{proof}

\begin{lemma} \label{lemmaint}
 Let $N$ be an oriented 3-manifold. 
Let $C$ be a rational 3-chain and let $\Sigma_2$ and $\Sigma_3$ be rational 2-chains, pairwise transverse in $N$. Then:
$$<\partial C,\Sigma_2,\Sigma_3>=<C,\partial \Sigma_2,\Sigma_3>-<C,\Sigma_2,\partial \Sigma_3>.$$ 
\end{lemma}
\begin{proof}
 It suffices to prove the result for pairwise transverse integral chains. 
Since $$\partial (C\cap \Sigma_2\cap \Sigma_3)=\lbp\partial C\cap \Sigma_2\cap \Sigma_3\rbp\cup
\lbp C\cap\partial (\Sigma_2\cap \Sigma_3)\rbp,$$ 
we have $<\partial C,\Sigma_2,\Sigma_3>=-<C,\partial (\Sigma_2\cap \Sigma_3)>$. 
Now, $$\partial (\Sigma_2\cap \Sigma_3)=\lbp -\partial (\Sigma_2)\cap \Sigma_3\rbp \cup 
\lbp \Sigma_2\cap \partial (\Sigma_3)\rbp.$$ The announced equality follows.
\end{proof}

\begin{corollary} \label{subinteq}
 Let $(M,K)$ be a $\Q$SK-pair. Let $\tilde{X}$ be the associated infinite cyclic covering. 
Let $C$ be a rational 3-chain and let $\Sigma_2$ and $\Sigma_3$ be rational 2-chains, pairwise $\tau$-transverse in $\tilde{X}$. Then:
$$<\partial C,\Sigma_2,\Sigma_3>_e=<C,\partial \Sigma_2,\Sigma_3>_e-<C,\Sigma_2,\partial \Sigma_3>_e.$$ 
\end{corollary}
\begin{proof}
 Apply Lemma \ref{lemmaint} to $C$, $\tau^{k_2}\Sigma_2$ and $\tau^{k_3}\Sigma_3$ for all integers $k_2$, $k_3$.
\end{proof}

\proofof{Lemma \ref{lemmaindsurf}}
Replace $\Sigma_1$ by a chain $\Sigma_1'$ satisfying the same conditions. Lemma \ref{subh2} shows that there is a rational 
3-chain $C$ such that $\partial C=\Sigma_1'-\Sigma_1$. Compute the difference 
$$<\Sigma_1',\Sigma_2,\Sigma_3>_e-<\Sigma_1,\Sigma_2,\Sigma_3>_e=<\partial C,\Sigma_2,\Sigma_3>_e.$$ By Corollary \ref{subinteq}, 
$$<\partial C,\Sigma_2,\Sigma_3>_e=<C,\partial \Sigma_2,\Sigma_3>_e-<C,\Sigma_2,\partial \Sigma_3>_e.$$
Hence, by Lemma \ref{lemmaformulae}:
$$<\partial C,\Sigma_2,\Sigma_3>_e=P_2(t_2)<C,\mu_2,\Sigma_3>-P_3(t_3)<C,\Sigma_2,\mu_3>,$$
and this is trivial in $\displaystyle \frac{\Qti}{(t_1t_2t_3-1,P_1(t_1),P_2(t_2),P_3(t_3))}$.

Let $\mu_1'$ be a knot in $\tilde{X}$, homologous to $\mu_1$, whose image in M is disjoint from the images of $\mu_2$ and $\mu_3$. 
The difference $\mu_1'-\mu_1$ is trivial in $H_1(\tilde{X};\Q)$, hence there is a rational 2-chain $S$ 
such that $\partial S=\mu_1'-\mu_1$. Choose $S$ $\tau$-transverse to $\Sigma_1$, $\Sigma_2$, and $\Sigma_3$. 
Set $\hat{S}=P_1(\tau)S$, and $\Sigma_1'=\hat{S}+\Sigma_1$. 
We have $\partial \Sigma_1' = P_1(\tau)\partial S+\partial \Sigma_1=P_1(\tau)\mu_1'$. Since: 
$$<\hat{S},\Sigma_2,\Sigma_3>_e= P_1(t_1) <S,\Sigma_2,\Sigma_3>_e = 0\quad \textrm{in}\quad \frac{\Qti}{(t_1t_2t_3-1,P_1(t_1),P_2(t_2),P_3(t_3))},$$
we have $<\Sigma_1',\Sigma_2,\Sigma_3>_e=<\Sigma_1,\Sigma_2,\Sigma_3>_e$.

Conclude by using the symmetry properties of the equivariant triple intersections.
\fin

  \section{Variation under null borromean surgeries} \label{secvar}

In this section, we prove Proposition \ref{propdiff}.

The following lemma describes the effect of a borromean surgery on the triple intersection numbers.
\begin{lemma} \label{lemmabd}
 Let $N$ be a 3-manifold. Let $\Gamma$ be a Y-graph in $N$, with leaves $\ell_1$, $\ell_2$, $\ell_3$. Let $\Sigma_1$, $\Sigma_2$, $\Sigma_3$, 
be transverse compact surfaces in $N$. Assume $\Gamma\cap\Sigma_i\cap\Sigma_j=\emptyset$ for $i\neq j$.
Then there are surfaces $\Sigma'_1$, $\Sigma'_2$, $\Sigma'_3$, in $N(\Gamma)$ such that $\partial \Sigma_i'=\partial \Sigma_i$, and:
$$<\hspace{-1pt}\Sigma'_1,\Sigma'_2,\Sigma'_3\hspace{-1pt}>_{\scriptscriptstyle N(\Gamma)}\hspace{-2pt}-\hspace{-2pt}<\hspace{-1pt}\Sigma_1,\Sigma_2,\Sigma_3\hspace{-1pt}>_{\scriptscriptstyle N} 
  =\sum_{\sigma\in\perm}\hspace{-2pt}\sg\hspace{-2pt}<\hspace{-1pt}\Sigma_1,\ell_{\sigma(1)}\hspace{-1pt}>_{\scriptscriptstyle N}
  <\hspace{-1pt}\Sigma_2,\ell_{\sigma(2)}\hspace{-1pt}>_{\scriptscriptstyle N}<\hspace{-1pt}\Sigma_3,\ell_{\sigma(3)}\hspace{-1pt}>_{\scriptscriptstyle N}\hspace{-3pt}.$$
\end{lemma}
\begin{proof}
The surgery replaces a tubular neighborhood $T(\Gamma)$ of $\Gamma$ by another standard handlebody of genus 3 (see Matveev \cite{Mat}). 
\begin{figure}[htb] 
\begin{center}
\begin{tikzpicture} [scale=0.3,fill opacity=0.5]

\newcommand{\creux}[2]{
\draw[xshift=#1,yshift=#2,thick] (-1.5,0.2) ..controls +(0.5,-0.7) and +(-0.5,-0.7) .. (1.5,0.2);
\draw[xshift=#1,yshift=#2,thick] (-1.2,0) ..controls +(0.6,0.4) and +(-0.6,0.4) .. (1.2,0);}
\creux{0}{-11cm}
\creux{9.5cm}{5.5cm}
\creux{-9.5cm}{5.5cm}

\newcommand{\bras}[1]{
\draw[rotate=#1,thick] (-2.5,-6.67) arc (-240:60:5);
\draw[rotate=#1] (0,-1.5) circle (2.5);
\draw [rotate=#1,white,line width=8pt] (-0.95,-4) -- (0.95,-4);
\draw[rotate=#1] {(0,-11) circle (3) (1,-3.85) -- (1,-7.6)};
\draw[rotate=#1,white,line width=6pt] (-1,-5) -- (-1,-8.7);
\draw[rotate=#1] {(-1,-3.85) -- (-1,-8.7) (-1,-8.7) arc (-180:0:1)};
\draw[rotate=#1,thick] (2.5,-6.67) ..controls +(-1.73,1) and +(-1.73,1) .. (7.03,1.17);}
\bras{0}
\draw [white,line width=10pt,rotate=120] (0,-1.5) circle (2.5);
\bras{120}
\draw [rotate=-120,white,line width=10pt] (-1.77,0.27) arc (135:190:2.5);
\draw [rotate=-120,white,line width=10pt] (1.77,0.27) arc (45:90:2.5);
\bras{-120}
\draw [white,line width=10pt] (-1.77,0.27) arc (135:190:2.5);
\draw [white,line width=10pt] (1.77,0.27) arc (45:90:2.5);
\draw (-1.77,0.27) arc (135:190:2.5);
\draw (1.77,0.27) arc (45:90:2.5);

\draw[gray] (0,-11.3) ..controls +(-1,-0.8) and +(-1,0.8) .. (0,-16);
\draw[gray] (0,-11.3) ..controls +(0.5,-0.3) and +(0,0.4) .. (0.75,-13.3) (0.7,-14.4) .. controls +(0,-0.3) and +(0.5,0.3) .. (0,-16);
\draw[gray] (0,-13.6) ..controls +(-0.17,-0.13) and +(-0.17,0.13) .. (0,-14.4);
\draw[gray] (0,-13.6) ..controls +(0.17,-0.13) and +(0.17,0.13) .. (0,-14.4);
\draw[gray] (0,-7.6) ..controls +(-0.17,-0.13) and +(-0.17,0.13) .. (0,-8.4);
\draw[gray] (0,-7.6) ..controls +(0.17,-0.13) and +(0.17,0.13) .. (0,-8.4);
\draw[gray] (0,-13.6) arc (-90:90:2.6);
\draw[gray] (0,-14.4) arc (-90:90:3.4);
\draw[gray] (0,-1.5) circle (2.3);
\draw [rotate=120,white,line width=10pt] (-1.77,0.27) arc (135:190:2.5);
\draw [rotate=120,white,line width=10pt] (1.77,0.27) arc (45:90:2.5);
\draw [rotate=120] (-1.77,0.27) arc (135:190:2.5);
\draw [rotate=120] (1.77,0.27) arc (45:90:2.5);
\draw [white,line width=8pt] (-0.75,-3.8) -- (0.75,-3.8);
\draw[gray] (0.8,-3.65) -- (0.8,-7.6) (-0.8,-3.65) -- (-0.8,-8.7) (-0.8,-8.7) arc (-180:0:0.8);
\fill[gray!40] (0,-1.5) circle (2.3);
\fill[gray!40] (-0.8,-3.65) -- (-0.8,-8.7) arc (-180:0:0.8) -- (0.8,-3.65) -- (-0.8,-3.65);
\draw[gray] (-1.7,-1.7) circle (0.3);
\draw[gray] (1,-0.3) circle (0.3);
\draw[gray] (-1.7,-1.4) arc (-100:-30:2.3);
\draw[gray] (-1.7,-2) arc (-100:-25:2.7);
\fill[gray!80] (-1.7,-2) arc (-100:-25:2.7) (1.26,-0.45) ..controls +(0.15,0.26) and +(0.15,0.26) .. (0.8,-0.1) 
  ..controls +(-0.6,-1.02) and +(1.2,0) .. (-1.7,-1.4) ..controls +(-0.35,0) and +(-0.35,0) .. (-1.7,-2);
\fill[gray!40] (0,-11.3) ..controls +(-1,-0.8) and +(-1,0.8) .. (0,-16) .. controls +(1,0.8) and +(1,-0.8) .. (0,-11.3);
\fill[gray!80] (0,-13.6) ..controls +(-0.17,-0.13) and +(-0.17,0.13) .. (0,-14.4) (0,-14.4) arc (-90:90:3.4) 
  (0,-7.6) ..controls +(-0.17,-0.13) and +(-0.17,0.13) .. (0,-8.4) ..controls +(3.4,0) and +(3.4,0) .. (0,-13.6);

\draw[->,gray,thick] (-0.7,-14.4) -- (-0.68,-14.5);
\draw[->,thick] (-0.78,-1.36) -- (-0.88,-1.39);
\draw (-0.48,-1.27) -- (-0.88,-1.39);

\draw[very thin] (3.8,0.75) arc (0:-105:2.5);
\draw[very thin] (-1.2,0.75) arc (-180:-138:2.5);

\draw[dashed] (-1.7,-1.4) arc (-100:-30:2.3);
\draw[dashed] (-1.7,-1.4) .. controls +(0.5,1) and +(-1.2,0.2) .. (0.8,-0.1);

\end{tikzpicture}
\end{center}
\caption{Surface in the reglued handlebody} \label{figsurface}
\end{figure}
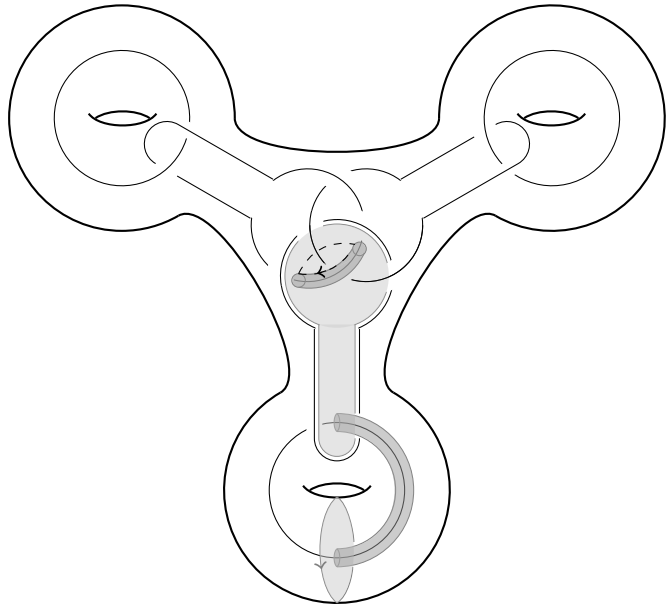
To each intersection point of a leaf $\ell_i$ with a surface $\Sigma_j$ corresponds 
a disk on $\Sigma _j$ which is removed by the surgery. It can be replaced, after surgery, with the surface inside the reglued handlebody 
drawn in Figure \ref{figsurface}, where the apparent boundary inside the handlebody bounds a disk in the corresponding reglued torus. 
Let $F_2$ denote the surface drawn in Figure \ref{figsurface}, and let $F_1$ (resp. $F_3$) be the similar surface corresponding 
to the left (resp. right) handle. Then the dashed curve represents the intersection $F_1\cap F_2$, and we have $<F_1,F_2,F_3>=1$. 
We obtain the result by counting the intersection points inside the reglued handlebody.
\end{proof}

\proofof{Proposition \ref{propdiff}}
Thanks to $\Q$-linearity, it suffices to prove the result for integral homology classes $\beta_j$. 
Consider representatives $\mu_j$ of the $\beta_j$ whose images in $M\setminus K$ are pairwise disjoint and disjoint from $\Gamma$. 
Consider $\tau$-transverse rational 2-chains $\Sigma_j$, $\tau$-transverse to $\tilde{\Gamma}$, such that $\partial \Sigma_j=\delta(\tau)\mu_j$, 
and $\tilde{\Gamma}\cap\tau^{k_i}\Sigma_i\cap\tau^{k_j}\Sigma_j=\emptyset$ for $i\neq j$ and $k_i,k_j\in\Z$. 
The surgery on $\Gamma$ gives rise to simultaneous surgeries on all the $\tau^k\tilde{\Gamma}$ in $\tilde{X}$. 
Hence, by Lemma \ref{lemmabd}:
\begin{eqnarray*}
&& \hspace{-1cm} \phi^{(M,K,\xi)(\Gamma)}([\mu_1]\otimes[\mu_2]\otimes[\mu_3])-\phi^{(M,K,\xi)}([\mu_1]\otimes[\mu_2]\otimes[\mu_3]) \\
 &=& \sum_{k_2,k_3\in\Z} \sum_{k\in\Z} \sum_{\sigma\in\perm} \sg <\Sigma_1,\tau^k\gamma_{\sigma(1)}>
  <\tau^{-k_2}\Sigma_2,\tau^k\gamma_{\sigma(2)}><\tau^{-k_3}\Sigma_3,\tau^k\gamma_{\sigma(3)}> t_2^{k_2}t_3^{k_3} \\
 &=& \sum_{\sigma\in\perm}\sg\sum_{k\in\Z} <\Sigma_1,\tau^k\gamma_{\sigma(1)}>lk_e(\delta(\tau)\mu_2,\tau^k\gamma_{\sigma(2)})(t_2)
   \,lk_e(\delta(\tau)\mu_3,\tau^k\gamma_{\sigma(3)})(t_3) \\
 &=& \sum_{\sigma\in\perm}\sg\lbp\sum_{k\in\Z} <\Sigma_1,\tau^k\gamma_{\sigma(1)}>t_1^k\rbp 
  \delta(t_2)lk_e(\mu_2,\gamma_{\sigma(2)})(t_2) \,\delta(t_3)lk_e(\mu_3,\gamma_{\sigma(3)})(t_3) \\
 &=& \sum_{\sigma\in\perm}\sg\prod_{j=1}^3 \delta(t_j) lk_e(\mu_j,\gamma_{\sigma(j)})(t_j)
\end{eqnarray*}
\fin

    \section{Structure of $\HA$} \label{secH}

In this section, we study the structure of $\Ah$ and $\HA$, and we prove Theorems \ref{thstructureAh} and \ref{thstructureh}. 

There is a natural surjective map $\displaystyle \Ah\twoheadrightarrow\HA$, 
which splits into surjective maps $\Al(\mi)\twoheadrightarrow\HA(\mi)$ for $\mi\in\{1,..,q\}^3$. Note that the map 
$\Al(\mi)\twoheadrightarrow\HA(\mi)$ is an isomorphism if and only if the $i_j$ are all distinct. 

Theorem \ref{thstructureAh} will follow from Lemma \ref{lemmaA} below. 

For $1\leq i\leq q$, $\C\otimes\Al_i$ can be written:
$$\C\otimes\Al_i=\bigoplus_{\ell=1}^{q_i} \frac{\Ct}{((t-z_{i\ell})^{m_i})}\eta_{i\ell},$$
where the $z_{i\ell}$ are complex roots of $\delta_i$, different from $0$ and $1$. 
Set: $$J_{\mi}=\{1,..,q_{i_1}\}\times\{1,..,q_{i_2}\}\times\{1,..,q_{i_3}\}.$$ Let $\ml=(\ell_j)_{1\leq j\leq 3}\in J_{\mi}$. 
Let $\Al(\mi,\ml)$ be the quotient of $\displaystyle \bigotimes_{1\leq j\leq 3} \frac{\Ct}{((t-z_{i_j\ell_j})^{m_{i_j}})}\eta_{i_j\ell_j}$ 
by the vector subspace generated by the holonomy relations, namely the relations 
$\otimes_{1\leq j\leq 3} \beta_j =\otimes_{1\leq j\leq 3} t\beta_j$.
Then $\C\otimes\Al(\mi)=\bigoplus_{\ml\in J_{\mi}} \Al(\mi,\ml).$

\begin{lemma} \label{lemmaA}
 The complex vector space $\Al(\mi,\ml)$ is non trivial if and only if $\prod_{j=1}^3 z_{i_j\ell_j}=1$. 
\end{lemma}
The following sublemma will be useful for rewriting the holonomy relations. 
\begin{sublemma} \label{sublemmahol}
 For all $(\beta_j)_{1\leq j\leq 3}\subset\Al$, for all $(w_j)_{1\leq j\leq 3}\subset\C$: 
$$\otimes_{1\leq j\leq 3} t\beta_j=\sum_{I\subset \{1,2,3\}} \otimes_{1\leq j\leq 3} p_I(w_j)\beta_j,$$
where $p_I(w_j)=\left\lbrace \begin{array}{ll} (t-w_j) & \textrm{ if } j\in I \\ w_j & \textrm{ if } j\notin I \end{array} \right.$.
\end{sublemma}
\begin{proof}
For $1\leq j\leq 3$, write $t=(t-w_j)+w_j$.
\end{proof}
\proofof{Lemma \ref{lemmaA}}
Fix $(\mi,\ml)$, and simplify the notation by setting $z_j=z_{i_j\ell_j}$, $n_j=m_{i_j}$, and for $\mk=(k_j)_{1\leq j\leq 3}\in\N^3$, 
$[\mk]=\bigotimes_{1\leq j\leq 3}(t-z_j)^{k_j}\eta_{i_j\ell_j}$. 
Thanks to Sublemma \ref{sublemmahol}, the holonomy relations can be written in terms of these generators, as follows:
$$hol(\mk) : \quad [\mk]=
\sum_{I\subset\{1,2,3\}} (\prod_{j\notin I} z_j) [\mk+\md],$$
where $(\delta_I)_j=\left\lbrace \begin{array}{ll} 1 & \textrm{ if } j\in I \\ 0 & \textrm{ if } j\notin I \end{array} \right.$. 
We have: $$\Al(\mi,\ml)=\frac{\C<[\mk]; 0\leq k_j< n_j\ \forall j>}{\C< hol(\mk); 0\leq k_j< n_j\ \forall j>}.$$

First assume $z_1z_2z_3\neq1$. For $\mk=(k_1,k_2,k_3)$, let $s(\mk)=k_1+k_2+k_3$. 
By decreasing induction on $s(\mk)$, we will prove that all the $[\mk]$ vanish in $\Al(\mi,\ml)$. 
It is true if $s(\mk)>n_1+n_2+n_3-2$. Fix $s\geq 0$, and assume $[\mk]=0$ if $s(\mk)>s$. 
Then, if $s(\mk)=s$, the relation $hol(\mk)$ becomes $[\mk]=(z_1z_2z_3)[\mk]$, hence $[\mk]=0$.

Now assume $z_1z_2z_3=1$. In this case, the holonomy relations get simplified:
$$hol(\mk) : \quad 
\sum_{\emptyset\neq I\subset\{1,2,3\}} (\prod_{j\notin I} z_j) [\mk+\md]=0.$$
The generator $[(0,0,0)]$ does not appear in any of these relations. Hence $\Al(\mi,\ml)\neq 0$.
\fin

\paragraph{Examples}
\begin{enumerate}
 \item Let $\displaystyle \Al=\frac{\Qt}{(t^4+1)}\eta_1\oplus\frac{\Qt}{(t^2+1)}\eta_2$. Let $\zeta=e^{i\frac{\pi}{4}}$. Then:
$$\C\otimes\Al=\frac{\Ct}{(t-\zeta)}\eta_{11}\oplus\frac{\Ct}{(t-\zeta^3)}\eta_{12}\oplus\frac{\Ct}{(t+\zeta)}\eta_{13}\oplus
\frac{\Ct}{(t+\zeta^3)}\eta_{14}\oplus\frac{\Ct}{(t-i)}\eta_{21}\oplus\frac{\Ct}{(t+i)}\eta_{22}$$
The space $\Al(\mi,\ml)$ is non trivial if and only if the set $\{(i_1,l_1),(i_2,l_2),(i_3,l_3)\}$ is, up to permutation, 
one of the following ones: $\{(1,1),(1,3),(2,1)\}$, $\{(1,2),(1,2),(2,1)\}$, $\{(1,4),(1,4),(2,1)\}$, $\{(1,1),(1,1),(2,2)\}$, 
$\{(1,2),(1,4),(2,2)\}$, $\{(1,3),(1,3),(2,2)\}$. There are 24 different non trivial $\Al(\mi,\ml)$, and each has complex dimension 1, 
hence $\dim_\Q(\Ah)=24$.
 \item Let $\displaystyle \Al=\frac{\Qt}{((t+1+t^{-1})^m)}$, $m>0$. In this case, $q=1$ and $\Ah=\Al(1,1,1)$. 
Over the complex numbers, we have $\displaystyle \C\otimes\Al=\frac{\Ct}{((t-j)^m)}\eta_{11}\oplus\frac{\Ct}{((t-j^2)^m)}\eta_{12}$, 
and $\displaystyle \C\otimes\Ah=\Al((1,1,1),(1,1,1))\oplus\Al((1,1,1),(2,2,2))$, where both the two components of this direct sum 
are non trivial. In particular, $\Ah$ has dimension at least 2.
\end{enumerate}

\proofof{Theorem \ref{thstructureh}}
Fix $\mi\in\{1,..,q\}^3$ such that $i_1\leq i_2\leq i_3$ and $\ml\in J_{\mi}$. Set $z_j=z_{i_j\ell_j}$, $n_j=m_{i_j}$, 
and $\eta_j=\eta_{i_j\ell_j}$. For $\mk=(k_j)_{1\leq j\leq 3}\in\N^3$, set $[\mk]_{\HA}=(t-z_1)^{k_1}\eta_1\wedge(t-z_2)^{k_2}\eta_2\wedge(t-z_3)^{k_3}\eta_3$. 
Let $\HA(\mi,\ml)$ denote the complex vector subspace of $\C\otimes_\Q\HA$ generated by the $[\mk]_{\HA}$. 
Note that:
$$\C\otimes_\Q\HA(\mi)=\bigoplus_{\ml\in J^{o}_{\mi}} \HA(\mi,\ml),$$
where $J^{o}_{\mi}$ is the set of all $\ml$ in $J_{\mi}$ such that, for $j=1,2$, if $i_j=i_{j+1}$, then $\ell_j\leq\ell_{j+1}$. 
Assume $\ml\in J^{o}_{\mi}$. We shall prove that $\HA(\mi,\ml)\neq 0$ if and only if $z_1z_2z_3=1$ and for $1\leq j\leq 3$, 
$n_j$ is at least the number of occurrences of $(i_j,\ell_j)$ in $((i_1,\ell_1),(i_2,\ell_2),(i_3,\ell_3))$. 
If $z_1z_2z_3\neq 1$, $\Al(\mi,\ml)=0$ implies $\HA(\mi,\ml)=0$. For the end of the proof, assume $z_1z_2z_3=1$. 
In this case, note that the holonomy relation $hol(\mk)$ relates generators $[\mk']_{\HA}$ such that $s(\mk')>s(\mk)$. 

If the $(i_j,l_j)$ are all distinct, then $\HA(\mi,\ml)\cong\Al(\mi,\ml)\neq 0$. 

Assume $(i_1,\ell_1)=(i_2,\ell_2)\neq(i_3,\ell_3)$. If $n_1=n_2=1$, the anti-symmetry implies $\HA(\mi,\ml)=0$. 
Otherwise $n_1=n_2\geq 2$. In this case, the space $\HA(\mi,\ml)$ is defined by the generators $[\mk]_{\HA}$ 
with $k_1<k_2$ and the holonomy relations $hol(\mk)$ with $k_1<k_2$, rewritten in terms of these generators. Indeed, 
a relation $hol(k_1,k_1,k_3)$ is trivial, and a relation $hol(k_2,k_1,k_3)$ is equivalent to $hol(k_1,k_2,k_3)$. 
The generator $[0,1,0]_{\HA}$ is non trivial since it does not appear in any relation $hol(\mk)$ with $k_1<k_2$.
The proof is the same whenever there are exactly two different $(i_j,\ell_j)$.

Assume $(i_1,\ell_1)=(i_2,\ell_2)=(i_3,\ell_3)$. If $n_1=n_2=n_3\leq2$, then $\HA(\mi,\ml)=0$. 
Otherwise $n_1=n_2=n_3\geq 3$. In this case, the space $\HA(\mi,\ml)$ is defined by the generators $[\mk]_{\HA}$ 
with $k_1<k_2<k_3$ and the holonomy relations $hol(\mk)$ with $k_1<k_2<k_3$, rewritten in terms of these generators. 
The generator $[0,1,2]_{\HA}$ does not appears in any of these relations. Hence $\HA(\mi,\ml)\neq0$. 
\fin

\paragraph{Examples}
\begin{enumerate}
 \item For $\displaystyle \Al=\frac{\Qt}{(t^4+1)}\eta_1\oplus\frac{\Qt}{(t^2+1)}\eta_2$, we have: 
$$\C\otimes\HA=\HA((1,1,2),(1,3,1))\oplus\HA((1,1,2),(2,4,2)),$$ and $\dim(\HA)=2$.
 \item For $\displaystyle \Al=\frac{\Qt}{((t+1+t^{-1})^m)}$, $\HA$ is trivial if $m\leq 2$. If $m\geq 3$, 
$$\C\otimes\HA=\HA((1,1,1),(1,1,1))\oplus\HA((1,1,1),(2,2,2)),$$ with both components non trivial. Hence $\HA$ has dimension 
at least 2.
\end{enumerate}

In the remaining of the section, we further study the structure of $\Al(\mi,\ml)$, and we provide bounds for the dimension of $\HA$.
\begin{lemma} \label{lemmadima}
Fix $(\mi,\ml)$, and simplify the notation by setting $z_j=z_{i_j\ell_j}$, $n_j=m_{i_j}$, and for $\mk=(k_j)_{1\leq j\leq 3}\in\N^3$, 
$[\mk]=\bigotimes_{1\leq j\leq 3}(t-z_j)^{k_j}\eta_{i_j\ell_j}$. 
Assume $z_1z_2z_3=1$. Assume $n_1\geq n_2\geq n_3$. Then the vector space $\Al(\mi,\ml)$ is generated by the family 
$([0,k_2,k_3])_{0\leq k_j<n_j}$. If $n_2+n_3\leq n_1+1$, this family is a basis of $\Al(\mi,\ml)$, and hence 
$\dim_\C \Al(\mi,\ml)=n_2n_3$. If $n_2+n_3>n_1+1$, then 
$n_2n_3-\frac{1}{2}(n_2+n_3-n_1)(n_2+n_3-n_1-1)\leq\dim_\C \Al(\mi,\ml)\leq n_2n_3$.
\end{lemma}
Note that if the $(i_j,l_j)$ are all distinct, then $\HA(\mi,\ml)\cong\Al(\mi,\ml)$, and the above statements hold for $\HA(\mi,\ml)$. 
\begin{sublemma} \label{sublemmadim}
 If $s\geq n_2+n_3-1$, the following equivalence holds:
$$\lbp hol(\mk) \textrm{ for all } \mk \textrm{ such that } s(\mk)\geq s-1\rbp \Leftrightarrow 
\lbp [\mk]=0 \textrm{ for all } \mk \textrm{ such that } s(\mk)\geq s\rbp.$$
\end{sublemma}
\begin{proof}
We proceed by decreasing induction on $s$. For $s>n_1+n_2+n_3-2$, the result is trivial. 
Fix $s$ such that $n_2+n_3-1\leq s\leq n_1+n_2+n_3-2$. Let $\mk=(k_1,k_2,k_3)$ satisfy $s(\mk)=s$. 
If $k_1=0$, the condition on $s$ implies $[\mk]=0$. Assume $k_1>0$. Consider the relation:
$$hol(k_1-1,k_2,k_3) : \quad z_2z_3[k_1,k_2,k_3]+z_1z_3[k_1-1,k_2+1,k_3]+z_1z_2[k_1-1,k_2,k_3+1]=0.$$
By increasing induction on $k_1$, we can replace this relation by $[k_1,k_2,k_3]=0$. This uses all the relations $hol(\mk)$ for $s(\mk)=s-1$, 
except the relations $hol(n_1-1,k_2,k_3)$, but those are trivial.
\end{proof}
\proofof{Lemma \ref{lemmadima}}
Let $V(s)$ be the complex vector subspace of $\Al(\mi,\ml)$ generated by the $[0,h_2,h_3]$ such that $h_2+h_3\geq s$. 
Again by decreasing induction on $s$, we prove that for $s\leq n_2+n_3-2$, $[\mk]\in V(s)$. 
Fix $s$ such that $0< s\leq n_2+n_3-2$. Consider $\mk=(k_1,k_2,k_3)$ such that $s(\mk)=s$ and $k_1>0$. 
By the induction hypothesis, the relation $hol(k_1-1,k_2,k_3)$ implies:
$$z_2z_3[k_1,k_2,k_3]+z_1z_3[k_1-1,k_2+1,k_3]+z_1z_2[k_1-1,k_2,k_3+1]\in V(s+1).$$
Conclude by increasing induction on $k_1$. 

We have seen that the relation $hol(k_1-1,k_2,k_3)$ expresses $[k_1,k_2,k_3]$ in terms of the $[0,h_2,h_3]$. These generators 
$[0,h_2,h_3]$ may be related by the relations $hol(n_1-1,k_2,k_3)$. If $n_2+n_3\leq n_1+1$, there are no relations 
$hol(n_1-1,k_2,k_3)$ such that $n_1-1+k_2+k_3<n_2+n_3-2$. If $n_2+n_3>n_1+1$, an easy computation shows that there are 
$\frac{1}{2}(n_2+n_3-n_1)(n_2+n_3-n_1-1)$ pairs $(k_2,k_3)$ of integers such that $0\leq k_i<n_i$ and $n_1-1+k_2+k_3<n_2+n_3-2$ 
(note that this last condition implies $k_i<n_i$ for $i=2,3$).
\fin

Let $\NT$ be the set of all $(\mi,\ml)$ such that $1\leq i_1\leq i_2\leq i_3\leq q$, $\ml\in J_{\mi}^{o}$, 
$z_{i_1\ell_1}z_{i_2\ell_2}z_{i_3\ell_3}=1$, and for $j=1,2,3$, the multiplicity $m_{i_j}$ is a least the number of occurences 
of $(i_j,\ell_j)$ in $((i_1,\ell_1),(i_2,\ell_2),(i_3,\ell_3))$. 
By Theorem \ref{thstructureh}:
$$\C\otimes\HA=\sum_{(\mi,\ml)\in\NT}\HA(\mi,\ml).$$
Recall that if $i\leq i'$, $m_i\geq m_{i'}$.
\begin{theorem} \label{thdimH}
 For $(\mi,\ml)=((i_1,i_2,i_3),(\ell_1,\ell_2,\ell_3))\in\NT$, set 
$b(\mi,\ml)=m_{i_2}m_{i_3}-\frac{1}{2}(m_{i_2}+m_{i_3}-m_{i_1})(m_{i_2}+m_{i_3}-m_{i_1}-1)$ if the $(i_j,\ell_j)$ are all distinct 
and $m_{i_2}+m_{i_3}\leq m_{i_1}+1$, $b(\mi,\ml)=m_{i_2}m_{i_3}$ if the $(i_j,\ell_j)$ are all distinct and $m_{i_2}+m_{i_3}>m_{i_1}+1$, 
$b(\mi,\ml)=1$ otherwise. Set:
$$B(\mi,\ml)=\left\{ \begin{array}{ll} m_{i_2}m_{i_3} & if\ the\ (i_j,\ell_j)\ are\ all\ distinct, \\
m_{i_3}(m_{i_1}-1) & if\ (i_1,\ell_1)=(i_2,\ell_2)\neq (i_3,\ell_3), \\
\frac{1}{2}m_{i_2}(m_{i_2}-1) & if\ (i_1,\ell_1)\neq(i_2,\ell_2)=(i_3,\ell_3), \\
\frac{1}{2}(m_{i_1}-1)(m_{i_1}-2) & if\ (i_1,\ell_1)=(i_2,\ell_2)=(i_3,\ell_3).\end{array}\right.$$
Then: 
$$\sum_{(\mi,\ml)\in\NT}b(\mi,\ml)\leq\dim_\Q(\HA)\leq\sum_{(\mi,\ml)\in\NT}B(\mi,\ml).$$
\end{theorem}
\begin{proof}
We want to bound the dimension of $\HA(\mi,\ml)$. If the $(i_j,\ell_j)$ are all distinct, 
this is done in Lemma \ref{lemmadima}. In the other cases, the non-triviality is given by Theorem \ref{thstructureh}, and it remains 
to compute the upper bound.

First assume that $(i_1,l_1)=(i_2,l_2)\neq(i_3,l_3)$. In this case, Lemma \ref{lemmadima} and the anti-symmetry imply 
that $\HA(\mi,\ml)$ is generated by the $[0,k_2,k_3]_{\HA}$ such that $0<k_2<m_{i_2}$ and $0\leq k_3<m_{i_3}$. 
Hence $\dim(\HA(\mi,\ml))\leq m_{i_3}(m_{i_2}-1)$.

Now assume that $(i_1,l_1)\neq(i_2,l_2)=(i_3,l_3)$. Then $\HA(\mi,\ml)$ is generated by the $[0,k_2,k_3]_{\HA}$ such that $0\leq k_2<k_3<m_{i_2}$. 
Hence $\dim(\HA(\mi,\ml))\leq \frac{1}{2}m_{i_2}(m_{i_2}-1)$.

Finally assume that $(i_1,l_1)=(i_2,l_2)=(i_3,l_3)$. Then $\HA(\mi,\ml)$ is generated by the $[0,k_2,k_3]_{\HA}$ such that $0<k_2<k_3<m_{i_1}$. 
Hence $\dim(\HA(\mi,\ml))\leq \frac{1}{2}(m_{i_1}-1)(m_{i_1}-2)$.
\end{proof}

\paragraph{Examples}
\begin{enumerate}
 \item Let $\displaystyle \Al=\frac{\Qt}{((t^2+1)^3)}\eta_1\oplus\frac{\Qt}{((t+1)^2)}\eta_2$. Then:
$$\C\otimes\Al=\frac{\Ct}{((t-i)^3)}\eta_{11}\oplus\frac{\Ct}{((t+i)^3)}\eta_{12}\oplus\frac{\Ct}{((t+1)^2)}\eta_{21}.$$
The space $\Al(\mi,\ml)$ is non trivial for $\mi=(1,1,2)$ and $\ml=(1,1,1)$ or $(2,2,1)$. The treatement of both cases is the same. 
Lemma \ref{lemmadima} gives $5\leq \dim(\Al(\mi,\ml))\leq 6$. Moreover, the proof provides the following presentation:
$$\Al(\mi,\ml)=\frac{\C<[0,0,0],[0,0,1],[0,1,0],[0,1,1],[0,2,0],[0,2,1]>}{\C<hol(2,0,0)>}.$$
Writing down all the relations $hol(\mk)$ for $0\leq k_j<n_j$ and $s(\mk)=2$, we see that $hol(2,0,0)$ implies $[0,2,1]=0$. 
Finally $\dim_\C(\Al(\mi,\ml))=5$, and $\dim_\Q(\Ah)=10$. 

By Theorem \ref{thdimH}, $1\leq\dim(\HA(\mi,\ml))\leq 4$. Since $\HA(\mi,\ml)$ it is a quotient of $\Al(\mi,\ml)$, 
$[\mk]_{\HA}=0$ if $s(\mk)\geq 3$. Using the anti-symmetry, we obtain:
$$\HA(\mi,\ml)=\frac{\C<[0,1,0],[0,1,1],[0,2,0]>}{\C<hol(0,1,0)>}.$$
The relation $hol(0,1,0)$ implies $[0,1,1]_{\HA}=\pm i\,[0,2,0]_{\HA}$, hence $\dim_\C(\HA(\mi,\ml))=2$, and $\dim_\Q(\HA)=4$. 

 \item For $\displaystyle \Al=\frac{\Qt}{((t+1+t^{-1})^m)}$, we consider $\Al(\mi,\ml)$ for $\mi=(1,1,1)$ and $\ml=(1,1,1)$ or $(2,2,2)$. 
By Lemma \ref{lemmadima}: $$\frac{1}{2}m(m+1)\leq \dim(\Al(\mi,\ml))\leq m^2.$$ For $m>1$, this does not give the exact dimension. 
For low values of $m$, it can be computed by hand following the method of Lemma \ref{lemmadima}. The space $\Al(\mi,\ml)$ is generated 
by the $[0,k_2,k_3]$ up to the relations $hol(m-1,k_2,k_3)$ for $k_2+k_3\leq m-2$. We obtain:
$$\dim(\Al(\mi,\ml))=\left\{ \begin{array}{ll} 3 & if\ m=2 \\ 7 & if\ m=3 \\ 12 & if\ m=4 \end{array}\right. .$$
Now consider $\HA(\mi,\ml)$ for $m\geq3$. It is non trivial and of dimension at most $\frac{1}{2}(m-1)(m-2)$. 
Once again, the dimension can be computed by hand for low values of $m$. The same argument as in Sublemma \ref{sublemmadim} shows that 
$[\mk]_{\HA}=0$ if $s(\mk)\geq2m-2$. Hence $\HA(\mi,\ml)$ is generated by the $[\mk]_{\HA}$ with $0\leq k_1<k_2<k_3<m$ and $s(\mk)\leq2m-3$, 
up to the relations $hol(k_1,k_2,k_3)$ with $0\leq k_1<k_2<k_3<m$ and $s(\mk)\leq2m-4$.
\begin{center}
 $\begin{array}{|l|c|c|c|c|c|}
  \hline
  m & 3 & 4 & 5 & 6 & 7 \\
  \hline
  \dim(\HA(\mi,\ml)) & 1 & 1 & 2 & 3 & 4 \\
  \hline
 \end{array}$
\end{center}
\end{enumerate}

  \section{Decomposition and characterization of $\phi$} \label{seccar}

  \subsection{Realization of rational homology classes by knots} \label{subsecreal}

In this subsection, we prove Proposition \ref{propreal}, which will allow us to apply Proposition \ref{propdiff} with more efficiency in 
Subsection \ref{subphi}, in order to prove Theorem \ref{thcar}. 

Fix a Blanchfield module $(\Al,\bl)$.

\begin{definition} \label{defreal}
 Let $(M,K,\xi)\in\Ens^m(\Al,\bl)$. Let $\tilde{X}$ be the infinite cyclic covering associated with $(M,K)$. 
A homology class $\eta\in\Al$ is {\em realizable for $(M,K,\xi)$} if there is a knot $J$ in $\tilde{X}$ such that $[J]=\eta$.
\end{definition}

\begin{proposition} \label{propreal}
 Let $(M,K,\xi)\in\Ens^m(\Al,\bl)$. For all $\eta\in\Al$, there is a marked $\Q$SK-pair $(M',K',\xi')\in\Ens^m(\Al,\bl)$ 
such that $\phi^{(M',K',\xi')}=\phi^{(M,K,\xi)}$ and $\eta$ is realizable for $(M',K',\xi')$.
\end{proposition}
In order to prove this proposition, we introduce a specific kind of LP-surgeries. Recall LP-surgeries were defined in Subsection \ref{subsecgrad}.

\begin{definition}
 For $d\in\N\setminus\{0\}$, a {\em $d$-torus} is a rational homology torus $T_d$ such that 
there are simple closed curves $\alpha$, $\beta$ in $\partial T_d$, and $\gamma$ in $T_d$ which satisfy:
\begin{itemize}
 \item $<\alpha,\beta>_{\partial T_d}=1$,
 \item $H_1(\partial T_d;\Z)=\Z[\alpha]\oplus\Z[\beta]$,
 \item $H_1(T_d;\Z)=\frac{\Z}{d\Z}[\alpha]\oplus\Z[\gamma]$,
 \item $[\beta]=d[\gamma]$.
\end{itemize}
A {\em meridian of $T_d$} is a simple closed curve on $\partial T_d$ homologous to $\alpha$.\\
A {\em (null) $d$-surgery} is a (null) LP-surgery $(\frac{T_d}{T})$ where $T$ is a standard solid torus and $T_d$ is a $d$-torus.
\end{definition}
For any $d\in\N\setminus\{0\}$, there exists a $d$-torus (see \cite[lemma 2.5]{M2}).

\begin{lemma} \label{lemmadtorus}
Let $T_d$ be a $d$-torus. Let $m_1$, $m_2$, $m_3$ be disjoint meridians of $T_d$. There are rational 2-chains $S_1$, $S_2$, $S_3$ in $T_d$ such 
that $\partial S_j=d m_j$ for $j=1,2,3$. For any such chains, pairwise transverse, the triple intersection number $<S_1,S_2,S_3>$ is trivial.
\end{lemma}
\begin{proof}
 The existence of the $S_j$ is clear since $d[m_j]=0$ in $H_1(T_d;\Z)$. Let us check that $<S_1,S_2,S_3>$ does not depend on 
the choice of the $S_j$. Replace $S_1$ by a chain $S_1'$ satisfying the same conditions. Since $H_2(T_d;\Q)=0$, there is a rational 
3-chain $C$ such that $\partial C=S_1'-S_1$. We have: 
$$<S_1',S_2,S_3>-<S_1,S_2,S_3>=<\partial C,S_2,S_3>.$$ By Lemma \ref{lemmaint}:
$$<\partial C,S_2,S_3>=<C,\partial S_2,S_3>-<C,S_2,\partial S_3>.$$ Since $m_2\cap S_3=\emptyset$ and $S_2\cap m_3=\emptyset$, 
we obtain $<S_1',S_2,S_3>=<S_1,S_2,S_3>$. 

Let $N=[0,1]\times S^1\times S^1$ be a collar neighborhood of $\partial T_d$ in $T_d$, parametrized so that:
\begin{itemize}
 \item $\{1\}\times S^1 \times S^1=\partial T_d$,
 \item for $j=1,2,3$, $m_j=\{1\}\times S^1 \times \{z_j\}$ with $z_j\in S^1$.
\end{itemize}
Consider the 2-chains $S_j$ in the copy $\overline{T_d\setminus N}$ of $T_d$, so that:
\begin{itemize}
 \item for $j=1,2,3$, $\partial S_j=dm_j^0$, where $m^0_j=\{0\}\times S^1 \times \{z_j\}$.
\end{itemize}
For $j=1,2,3$, let $A_j$ be the annulus in $N$ whose slice is represented Figure \ref{figslice}. Set $S_j'=S_j+dA_j$. 
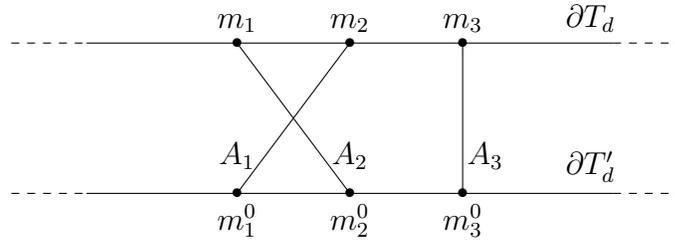
\begin{figure} [hbt]
\begin{center}
\begin{tikzpicture}
\foreach \y in {0,2} {
\draw (0,\y) -- (7,\y);
\draw[dashed] (-1,\y) -- (0,\y) (7,\y) -- (8,\y);}
\foreach \k in {1,2,3} {\foreach \y in {0,2} {
\draw (1.5*\k +0.5,\y) node {$\scriptstyle{\bullet}$};}
\draw (1.5*\k +0.5,0) node[below] {$m^0_{\k}$};
\draw (1.5*\k +0.5,2) node[above] {$m_{\k}$};}
\draw (2,0) -- (3.5,2);
\draw (2,0.5) node {$A_1$};
\draw (3.5,0) -- (2,2);
\draw (3.5,0.5) node {$A_2$};
\draw (5,0) -- (5,2);
\draw (5.3,0.5) node {$A_3$};
\draw (6.7,2) node[above] {$\partial T_d$};
\draw (6.7,0) node[above] {$\partial T'_d$};
\end{tikzpicture} \caption{Slices of the annuli $A_j$ in $N$.} \label{figslice}
\end{center}
\end{figure}
Since $\partial S_1'=dm_2$, $\partial S_2'=dm_1$, and $\partial S_3'=dm_3$, the independance with respect to the surfaces implies 
$<S'_2,S'_1,S'_3>=<S_1,S_2,S_3>$. But by construction, $<S'_1,S'_2,S'_3>=<S_1,S_2,S_3>$. Finally, $<S_1,S_2,S_3>=0$.
\end{proof}

\begin{lemma} \label{lemmainv}
 Null $d$-surgeries on marked $\Q$SK-pairs preserve the equivariant triple intersection map.
\end{lemma}
\begin{proof} 
 Let $(M,K,\xi)\in\Ens^m(\Al,\bl)$. Let $(\frac{T_d}{T})$ be a null $d$-surgery defined on $(M,K,\xi)$. Let $\tilde{T}$ be a lift 
of $T$ in the infinite cyclic covering $\tilde{X}$ associated with $(M,K)$. The infinite cyclic covering $\tilde{X}'$ associated with 
$(M,K,\xi)(\frac{T_d}{T})$ is obtained from $\tilde{X}$ by the surgeries $(\frac{T_d^{(k)}}{\tau^k(\tilde{T})})$ for all $k\in\Z$, 
where the $T_d^{(k)}$ are copies of $T_d$. Note that 
$\Al$ is generated over $\Q$ by the homology classes which are realizable by simple closed curves in 
$\tilde{X}\setminus\sqcup_{k\in\Z}\tau^k(\tilde{T})$. 
Hence it suffices to prove that the triple equivariant intersection is preserved for the homology classes of disjoint knots 
$\mu_1$, $\mu_2$, $\mu_3$ in $\tilde{X}\setminus\sqcup_{k\in\Z}\tau^k(\tilde{T})$. Let $\Sigma_1$, $\Sigma_2$, $\Sigma_3$, 
be $\tau$-transverse rational 2-chains, $\tau$-transverse to $\tilde{T}$, such that $\partial \Sigma_i=\delta(\tau)\mu_i$. 
Assume no $\tau$-translate of $\tilde{T}$ meets any of the pairwise intersections of the $\tau$-translates of the $\Sigma_i$. 
The 2-chains $\Sigma_i'=\Sigma_i\cap(\tilde{X}\setminus\sqcup_{k\in\Z}\tau^k(\tilde{T}))$ are preserved by the surgery. 
The boundary of $\Sigma_i'$ in $\tilde{X}'$ 
is the sum of $\delta(\tau)\mu_i$ and of a $\Q$-linear combination of meridians of the $T_d^{(k)}$. Use Lemma \ref{lemmadtorus} 
to add to the $\Sigma_i'$ rational 2-chains in the $T_d^{(k)}$ so that their boundaries reduce to $\delta(\tau)\mu_i$, without 
adding triple intersection points. 
\end{proof}

\proofof{Proposition \ref{propreal}}
Let $\eta\in\Al$. Let $d$ be a positive integer such that $d\eta$ is realizable for $(M,K,\xi)$. Let $\tilde{J}$ be a knot in the infinite 
cyclic covering $\tilde{X}$ associated with $(M,K)$, whose image $J$ in $M\setminus K$ is also a knot, and such that $[\tilde{J}]=d\eta$. 
Let $T(J)$ be a tubular neighborhood of $J$ which lifts into a tubular neighborhood $T(\tilde{J})$ of $\tilde{J}$. Let $T_d$ be a 
$d$-torus. Fix an LP-identification of $\partial T_d$ and $\partial T(J)$. Set $(M',K',\xi')=(M,K,\xi)(\frac{T_d}{T(J)})$. 
The covering $\tilde{X}'$ can be obtained from $\tilde{X}$ by simultaneous surgeries $(\frac{T_d^{(k)}}{\tau^k(T(\tilde{J}))})$, for all $k\in\Z$, 
where the $T_d^{(k)}$ are copies of $T_d$. 
Let $\gamma\subset T_d$ be a knot such that $d[\gamma]=[\ell(J)]$, where $\ell(J)$ is a parallel of $J$ in $\partial T(J)$ (which is preserved 
by the surgery). Note that all the parallels of $J$ have the same rational homology class, in $(M',K')$ as well as in $(M,K)$. 
Let $\tilde{\gamma}$ be the lift of $\gamma$ in $T_d^{(0)}$, so that $d[\tilde{\gamma}]=[\ell(\tilde{J})]$, where $\ell(\tilde{J})$ is the 
lift of $\ell(J)$ in $\partial T_d^{(0)}$. We have $[\tilde{\gamma}]=\eta$. Conclude with Lemma \ref{lemmainv}.
\fin

  \subsection{Study of the map $\phi$} \label{subphi}

In this subsection, we decompose the equivariant triple intersection map and we study the target spaces in order to prove Theorem \ref{thcar}. 
Fix a Blanchfield module $(\Al,\bl)$. Let $\delta$ be the normalized annihilator of $\Al$. Define 
a decomposition of $\Al$ and associated notation as in Subsection \ref{subsecH}. 

For $\mi\in\{1,..,q\}^3$, set:
$$\Rd(\mi)=\frac{\Qti}{(t_1t_2t_3-1,\delta_{i_1}(t_1),\delta_{i_2}(t_2),\delta_{i_3}(t_3))}.$$
Define a structure of $\Rd(\mi)$-module on $\Al(\mi)$ by: 
$$t_1^{k_1}t_2^{k_2}t_3^{k_3}.\otimes_{1\leq j\leq 3} \beta_j =\otimes_{1\leq j\leq 3} t^{k_j}\beta_j.$$ 
Then $\Al(\mi)$ is a free cyclic $\Rd(\mi)$-module generated by $\eta_{\mi}:=\eta_1\otimes\eta_2\otimes\eta_3$.

Lemmas \ref{lemmaformulae} and \ref{lemmaindsurf} imply:
\begin{proposition} \label{propinttriple}
 Let $(M,K,\xi)\in\Ens^m(\Al,\bl)$. Let $\tilde{X}$ 
be the infinite cyclic covering associated with $(M,K)$. Let $\mi=(i_1,i_2,i_3)\in\{1,..,q\}^3$. 
Define a $\Q$-linear map $\phi_{\mi}^{(M,K,\xi)} : \Al(\mi) \to \Rd(\mi)$ as follows. If $\mu_1$, $\mu_2$, $\mu_3$ are knots 
in $\tilde{X}$ whose images in $M\setminus K$ are pairwise disjoint, and such that $[\mu_j]\in\Al_{i_j}$ for $j=1,2,3$, 
let $\Sigma_1$, $\Sigma_2$, $\Sigma_3$ be $\tau$-transverse rational 2-chains such that $\partial \Sigma_j=\delta_{i_j}(\tau)\mu_j$, 
and set $$\phi_{\mi}^{(M,K,\xi)}([\mu_1]\otimes[\mu_2]\otimes[\mu_3])=<\Sigma_1,\Sigma_2,\Sigma_3>_e.$$
Then the map $\phi_{\mi}^{(M,K,\xi)}$ is well-defined and $\Rd(\mi)$-linear. 
\end{proposition}
When it does not seem to cause confusion, the map $\phi_{\mi}^{(M,K,\xi)}$ (resp. $\phi^{(M,K,\xi)}$) is denoted by $\phi_{\mi}$ 
(resp. $\phi$). Note that the maps $\phi_{\mi}$ depend on the decomposition of $\Al$ and on the normalisation of the $\delta_i$'s. 

It is easy to see that the maps $\phi_{\mi}$ and $\phi$ are related by:
$$\phi(\beta_1\otimes\beta_2\otimes\beta_3)=
\frac{\delta(t_1)\delta(t_2)\delta(t_3)}{\delta_{i_1}(t_1)\delta_{i_2}(t_2)\delta_{i_3}(t_3)}\phi_{\mi}(\beta_1\otimes\beta_2\otimes\beta_3),$$
for $\beta_1\otimes\beta_2\otimes\beta_3\in\Al(\mi)$.
This implies in particular that $\phi(\Al(\mi))$ is contained in the ideal of $\Rd_\delta$ generated by 
$\frac{\delta(t_1)\delta(t_2)\delta(t_3)}{\delta_{i_1}(t_1)\delta_{i_2}(t_2)\delta_{i_3}(t_3)}$. 
Let $\hat{\maps}$ be the set of all $\phi\in\maps$ which satisfy this condition. 
For any $\phi\in\hat{\maps}$, the above relation defines associated maps $\phi_{\mi}: \Al(\mi) \to \Rd(\mi)$. 

Note that the linearity implies that the map $\phi$ is encoded in the datum of the family of the $\phi(\eta_{\mi})$, or equivalently 
of the $\phi_{\mi}(\eta_{\mi})$. For $\mi$ fixed, the map $\phi_{\mi}$ is encoded in $\phi_{\mi}(\eta_{\mi})$. 

For $\mi$ such that the $i_j$ are all distinct, we will see below that any element of $\Rd(\mi)$ is a $\phi_{\mi}^{(M,K,\xi)}(\eta_{\mi})$ 
for some marked $\Q$SK-pair $(M,K,\xi)$. In general, the image may be restricted in the following sense. 
There is a surjective map $p_{\mi} : \Rd(\mi)\twoheadrightarrow\HA(\mi)$ given by 
$p_{\mi}(t_1^{k_1}t_2^{k_2}t_3^{k_3})=t^{k_1}\eta_{i_1}\wedge t^{k_2}\eta_{i_2}\wedge t^{k_3}\eta_{i_3}$. 
It corresponds to the natural projection $\Al(\mi)\twoheadrightarrow\HA(\mi)$ via the isomorphism $\Rd(\mi)\cong \Al(\mi)$ given by 
$t_1^{k_1}t_2^{k_2}t_3^{k_3}\mapsto t^{k_1}\eta_{i_1}\otimes t^{k_2}\eta_{i_2}\otimes t^{k_3}\eta_{i_3}$. Note that $\ker(p_{\mi})$ 
is not an ideal of $\Rd(\mi)$, and that we cannot define an $\Rd(\mi)$-module structure on $\HA(\mi)$ as we did on $\Al(\mi)$. 
The following lemma implies that we do not lose information when composing the map $\phi_{\mi}$ by the surjection $p_{\mi}$.

\begin{lemma}
 Let $\mi=(i_1,i_2,i_3)\in\{1,..,q\}^3$. There is rational vector subspace $\Rd(\mi)^a$ of $\Rd(\mi)$, which contains $\phi_{\mi}(\eta_{\mi})$, 
such that $p_{\mi}$ induces an isomorphism $\Rd(\mi)^a\cong\HA(\mi)$. 
\end{lemma}
\begin{proof}
If the $i_j$ are all distinct, the map $p_{\mi}$ is an isomorphism, and $\Rd(\mi)^a=\Rd(\mi)$. 
Assume the $i_j$ are not all distinct.

Set: $$\Sym=\{\sigma\in\perm\textrm{ such that } i_{\sigma(j)}=i_j \textrm{ for } j=1,2,3\}\subset\perm,$$
$$\Qti^a=\{P\in\Qti\,|\,P(t_{\sigma(1)},t_{\sigma(2)},t_{\sigma(3)})=\sg P(t_1,t_2,t_3)\ \forall\,\sigma\in\Sym\},$$
and let $\Qti^s$ be the rational vector subspace of $\Qti$ generated by the polynomials $P\in\Qti$ such that 
$P(t_{\tau(1)},t_{\tau(2)},t_{\tau(3)})=P(t_1,t_2,t_3)$ for some transposition $\tau\in\Sym$. 
\begin{sublemma}
 $\Qti=\Qti^s\oplus\Qti^a$
\end{sublemma}
\begin{proof}
Let $P\in\Qti$. Set: 
$$P^a(t_1,t_2,t_3)=\frac{1}{|\Sym|}\sum_{\sigma\in\Sym}\sg P(t_{\sigma(1)},t_{\sigma(2)},t_{\sigma(3)}),$$ where $|.|$ stands for the 
cardinality. We have $P^a\in\Qti^a$.

 We shall check that $\Qti^s\cap\Qti^a=0$, and that for $P\in\Qti$, $P-P^a$ is in $\Qti^s$. It is clear if $\Sym\neq\perm$. 
Assume $\Sym=\perm$. 

Let $P\in\Qti^s\cap\Qti^a$. Since $P\in\Qti^a$, $P=P^a$, and since $P\in\Qti^s$, $P=P_{12}+P_{13}+P_{23}$, where each $P_{ij}$ 
is invariant under the transposition $(ij)$. We have $P^a=P^a_{12}+P^a_{13}+P^a_{23}$, and each term in this sum is trivial. 
Hence $P=0$.

For $P(t_1,t_2,t_3)=t_1^{k_1}t_2^{k_2}t_3^{k_3}$, with $(k_1,k_2,k_3)\in\Z^3$, we have:
$$P(t_1,t_2,t_3)-P^a(t_1,t_2,t_3)=\frac{1}{6}(t_1^{k_1}t_2^{k_3}t_3^{k_2}+t_1^{k_3}t_2^{k_1}t_3^{k_2})
+\frac{1}{3}(t_1^{k_1}t_2^{k_2}t_3^{k_3}+t_1^{k_2}t_2^{k_1}t_3^{k_3})$$ 
$$\hspace{2cm} -\frac{1}{6}(t_1^{k_2}t_2^{k_3}t_3^{k_1}+t_1^{k_2}t_2^{k_1}t_3^{k_3})
-\frac{1}{3}(t_1^{k_3}t_2^{k_1}t_3^{k_2}+t_1^{k_3}t_2^{k_2}t_3^{k_1})+\frac{1}{2}(t_1^{k_1}t_2^{k_2}t_3^{k_3}+t_1^{k_3}t_2^{k_2}t_3^{k_1}).$$
In this expression, each parenthesized term is invariant under some transposition. 
Finally $\Qti=\Qti^s\oplus\Qti^a$.
\end{proof}

Let $\I$ be the ideal $(t_1t_2t_3-1,\delta_{i_1}(t_1),\delta_{i_2}(t_2),\delta_{i_3}(t_3))\subset\Qti$. We have:
$$\Rd(\mi)=\frac{\Qti}{\I}.$$ 
Set $\I^s=\I\cap\Qti^s$ and $\I^a=\I\cap\Qti^a$.
\begin{sublemma}
 $\I=\I^s\oplus\I^a$
\end{sublemma}
\begin{proof}
 It is clear that $\I^s\cap\I^a=0$. Let $P\in\I$. Writing $P$ as a combination of the generators of $\I$, we see that 
$P(t_{\sigma(1)},t_{\sigma(2)},t_{\sigma(3)})\in\I$ for all $\sigma\in\Sym$. Hence $P^a\in\I^a$, and it follows that $P-P^a\in\I^s$.
\end{proof}
We finally have the decomposition:
$$\Rd(\mi)=\Rd(\mi)^s\oplus \Rd(\mi)^a,$$
where $\displaystyle \Rd(\mi)^s=\frac{\Qti^s}{\I^s}$ and $\displaystyle \Rd(\mi)^a=\frac{\Qti^a}{\I^a}$. 
Since $\displaystyle \Rd(\mi)^a\cong\frac{\Rd(\mi)}{\Rd(\mi)^s}$ and $\Rd(\mi)^s\subset \ker(p_{\mi})$, we have the following 
commutative diagram of rational vector spaces, 
\begin{figure} [htb]
\begin{center}
\begin{tikzpicture} [yscale=0.6]
\draw (0,2) node {$\Rd(\mi)$};
\draw (2,2) node {$\Rd(\mi)^a$};
\draw (0,0) node {$\Al(\mi)$};
\draw (2,0) node {$\HA(\mi)$};
\draw[->>] (0.5,2) -- (1.4,2);
\draw[->>] (0.5,0) -- (1.5,0);
\draw[->] (0,1.5) -- (0,0.5);
\draw[->] (2,1.5) -- (2,0.5);
\draw (0,1) node[right] {$\cong$} (2,1) node[right] {$\cong$};
\draw[->>] (0.5,1.5) -- (1.5,0.5);
\draw (1.3,1.1) node {$p_{\mi}$};
\end{tikzpicture}
\end{center}
\end{figure}
where the isomorphism $\Rd(\mi)\fl{\scriptstyle{\cong}}\Al(\mi)$ is given by 
$t_1^{k_1}t_2^{k_2}t_3^{k_3}\mapsto t_1^{k_1}t_2^{k_2}t_3^{k_3}.\eta_{\mi}$. This isomorphism identifies $\Rd(\mi)^s$ with 
the subspace of $\Al(\mi)$ generated by the anti-symmetry relations. Hence $\displaystyle {p_{\mi}}_{|\Rd(\mi)^a}$ also is 
an isomorphism.

Relation (\ref{relperm}) implies $\phi_{\mi}(\eta_{\mi})\in \Rd(\mi)^a$. 
\end{proof}

The map $\phi$ is completely determined by the $\phi_{\mi}(\eta_{\mi})$ for $\mi=(i_1,i_2,i_3)$ such that $i_1\leq i_2\leq i_3$. 
Since $\HA$ is the direct sum of the $\HA(\mi)$ for these $\mi$, the above lemma implies that the datum of $\phi$ is finally encoded 
in the element $\h(\phi):=\sum_{\mi\in\NT} p_{\mi}\circ\phi_{\mi}(\eta_{\mi})$. This holds for any $\phi\in\hat{\maps}$, hence we obtain 
an injective map $\h : \hat{\maps}\hookrightarrow\HA$. Note that this map depends on the choice of a decomposition 
$\Al=\bigoplus_{1\leq i\leq q}\Al_i$. To obtain Theorem \ref{thcar}, it remains to prove the following lemma. 

\begin{lemma} \label{lemmasurj}
 The map $\h\circ\phi^{\bullet} : \Ens^m(\Al,\bl) \to \HA$ is surjective.
\end{lemma}
\begin{proof}
We prove that, for $\mi=(i_1,i_2,i_3)\in\{1,..,q\}^3$, any element of $\Rd(\mi)^a$ is equal to $\phi_{\mi}^{(M,K,\xi)}(\eta_{\mi})$ 
for some $(M,K,\xi)\in\Ens^m(\Al,\bl)$.

Let $(M,K,\xi)\in\Ens^m(\Al,\bl)$. We shall prove that, for any $r\in\Q$ and $(k_1,k_2,k_3)\in\Z^3$, there is a Y-graph $\Gamma$, 
null in $M\setminus K$, such that :
\begin{equation} \label{eq} 
\phi_{\mi}^{(M,K,\xi)(\Gamma)}(\eta_{\mi})-\phi^{(M,K,\xi)}_{\mi}(\eta_{\mi})=r\sum_{\sigma\in\Sym}\sg\prod_{1\leq j\leq 3}t_j^{k_{\sigma(j)}},
\end{equation}
where $\Sym=\{\sigma\in\perm\textrm{ such that } i_{\sigma(j)}=i_j \textrm{ for } j=1,2,3\}\subset\perm$. 
Since these differences generate $\Rd(\mi)^a$ as an additive group, this will prove that we can obtain any element of $\Rd(\mi)^a$. 

Fix $r\in\Q$ and $(k_1,k_2,k_3)\in\Z^3$. Let $\Gamma$ be a Y-graph, null in $M\setminus K$. Let $\tilde{\Gamma}$ be a lift of $\Gamma$ 
in the infinite cyclic covering $\tilde{X}$ associated with $(M,K)$. Let $\gamma_1$, $\gamma_2$, $\gamma_3$ be the homology classes in $\Al$ 
of the leaves of $\tilde{\Gamma}$, given in an order induced by the orientation of the internal vertex of $\Gamma$. 
By Proposition \ref{propdiff}:
$$\phi^{(M,K,\xi)(\Gamma)}_{\mi}(\eta_{\mi})-\phi^{(M,K,\xi)}_{\mi}(\eta_{\mi})=
\sum_{\sigma\in\perm}\sg\prod_{j=1}^3 \delta_{i_j}(t_j) \bl(\eta_{i_j},[\gamma_{\sigma(j)}])(t_j).$$
Set $\beta_j=t^{-k_j}a_{i_j}^{-1}(t^{-1})d(\eta_{i_j})$ for $j=1,2,3$, where the inverse 
of $a_{i_j}$ is defined modulo $\delta_{i_j}$. We want to choose $\Gamma$ such that $[\gamma_1]=r\beta_1$ and $[\gamma_j]=\beta_j$ 
for $j=2,3$. These homology classes may not be realizable for $(M,K,\xi)$. Use Proposition \ref{propreal} to replace $(M,K,\xi)$ with 
another marked $\Q$SK-pair (still denoted by $(M,K,\xi)$), so that the map $\phi$ remains unchanged, and the required homology classes 
are realizable. Then we can define $\Gamma$ as desired, and we obtain Equality (\ref{eq}).

This concludes since Proposition \ref{propdiff} implies that $\phi_{\mi'}(\eta_{\mi'})$ is modified by the surgery on $\Gamma$ if and only if 
$\mi'$ is a permutation of $\mi$.
\end{proof}

\section{Degree one invariants of marked $\Q$SK-pairs} \label{secG1}

\subsection{The borromean subquotient}

Fix a Blanchfield module $(\Al,\bl)$. 
Let $\F^{m,b}_1(\Al,\bl)$ be the rational vector subspace of $\F^m_1(\Al,\bl)$ generated by the brackets $[(M,K,\xi);\frac{B}{A}]$ 
where $(\frac{B}{A})$ is a borromean surgery. Let $\G^{m,b}_1(\Al,\bl)$ be the image of $\F^{m,b}_1(\Al,\bl)$ in the quotient 
$\G^m_1(\Al,\bl)$. In this subsection, we study $\G^{m,b}_1(\Al,\bl)$, and we prove:
\begin{proposition} \label{propborro}
 Set $\displaystyle \HA=\frac{\Lambda^3\Al}{(\beta_1\wedge\beta_2\wedge\beta_3=t\beta_1\wedge t\beta_2\wedge t\beta_3)}$. 
The map $\hs : \Ens^m(\Al,\bl) \to \HA$ of Proposition \ref{propdegree1} induces an isomorphism $\displaystyle \G^{m,b}_1(\Al,\bl)\cong\HA$.
\end{proposition}
The main point of the proof is the construction of a well-defined map $\varphi : \HA\to\G^{m,b}_1(\Al,\bl)$. 

Let $(M,K,\xi)\in\Ens^m(\Al,\bl)$. For a Y-graph $\Gamma$ null in $M\setminus K$, the bracket in $\F^{m,b}_1(\Al,\bl)$ 
associated with the surgery along $\Gamma$ is denoted by $[(M,K,\xi);\Gamma]$.

\newcommand{\diag}[7]{
\begin{tikzpicture} [scale=#7]
\draw (-1.4,-1) -- (0,0) -- (1.4,-1) (0,0) -- (0,2);
\draw (0,2.7) node {#1} (-2,-1.4) node {#2} (2,-1.4) node {#3};
\draw (-1.4,1) node {#4} (1.4,1) node {#5} (0,-2) node {#6};
\end{tikzpicture}}
Define a {\em Y-diagram} as a unitrivalent graph with one oriented trivalent vertex and three univalent vertices, equipped with the following 
labellings: \raisebox{-5ex}{\diag{$\beta_1$}{$\beta_2$}{$\beta_3$}{$f_{12}$}{$f_{13}$}{$f_{23}$}{0.4}}, 
where $\beta_i\in\Al$, and the $f_{ij}\in\Q(t)$ satisfy $f_{ij}\,mod\,\Qt=\bl(\beta_i,\beta_j)$.
In the pictures, the orientation of the trivalent vertex is given by the cyclic order  
\raisebox{-1.5ex}{\begin{tikzpicture} [scale=0.2]
\newcommand{\tiers}[1]{
\draw[rotate=#1,color=white,line width=4pt] (0,0) -- (0,-2);
\draw[rotate=#1] (0,0) -- (0,-2);}
\draw (0,0) circle (1);
\draw[<-] (-0.05,1) -- (0.05,1);
\tiers{0}
\tiers{120}
\tiers{-120}
\end{tikzpicture}}. 

We wish to realize Y-diagrams by Y-graphs in $M\setminus K$. Fix a ball $B\subset M\setminus K$ and a lift $\tilde{B}\subset\tilde{X}$ of this ball. 
Let $D$ be the above Y-diagram, and denote by $v_i$ the univalent vertex colored by $\beta_i$. Embed the diagram $D$ in the ball $B$, and equip it 
with the framing induced by an immersion in the plane which induces the fixed orientation of the internal vertex. At each univalent vertex $v_i$, 
glue a leaf $\ell_i$, trivial in $H_1(M\setminus K;\Q)$, in order to obtain a null Y-graph $\Gamma$. For all $i$, let $\hat{\ell}_i$ be the extension 
of $\ell_i$ in $\Gamma$ (see Figure \ref{figext}), 
\begin{figure}[htb]
\begin{center}
\begin{tikzpicture} [scale=0.4]
\begin{scope} [xshift=-5cm]
\draw (0,0) -- (0,3);
\draw (0,4) circle (1);
\draw (0,0) node {$\scriptstyle{\bullet}$};
\end{scope}
\draw[->,very thick] (-0.5,2.5) -- (0.5,2.5);
\begin{scope} [xshift=5cm]
\draw (0,0) .. controls +(0.2,0.05) and +(0,-3) .. (0.2,3);
\draw (0,0) .. controls +(-0.2,0.05) and +(0,-3) .. (-0.2,3);
\draw (0,4) circle (1);
\draw[white,very thick] (-0.18,3) -- (0.18,3);
\draw (0,0) node {$\scriptstyle{\bullet}$};
\end{scope}
\end{tikzpicture} 
\end{center}
\caption{Extension of a leaf in a Y-graph} \label{figext}
\end{figure}
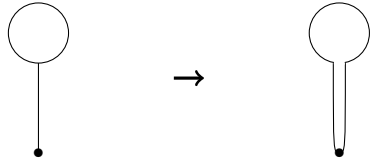
and let $\tilde{\ell}_i$ be the lift of $\hat{\ell}_i$ defined by lifting the basepoint in the ball $\tilde{B}$. 
The null Y-graph $\Gamma$ is a {\em realization of $D$ in $(M,K,\xi)$} if the following conditions are satisfied:
\begin{itemize}
 \item for all $i$, $[\tilde{\ell}_i]=\beta_i$, 
 \item for all $i<j$, $lk_e(\tilde{\ell}_i,\tilde{\ell}_j)=f_{ij}$. 
\end{itemize}
If such a realization exists, the Y-diagram $D$ is {\em realizable in $(M,K,\xi)$}. 
Note that the Y-diagram $D$ is realizable if and only if each $\beta_i$ is realizable for $(M,K,\xi)$ (see Definition \ref{defreal}). 

Let $(M,K,\xi)\in\Ens^m(\Al,\bl)$. Let $\F^{m,b}_2(M,K,\xi)$ be the subset of $\F^m_2(\Al,\bl)$ generated by the $[(M,K,\xi);\Gamma_1,\Gamma_2]$ 
for all Y-graphs $\Gamma_1$ and $\Gamma_2$ null in $M\setminus K$. 

\begin{lemma}[\cite{Mt} Chap. 6, Lemma 2.11] \label{lemmaindreal} 
 Let $(M,K,\xi)\in\Ens^m(\Al,\bl)$. Let $D$ be a Y-diagram. Let $\Gamma$ be a realization of $D$ in $(M,K,\xi)$. 
Then the class of $[(M,K,\xi);\Gamma]$ $mod\ \F_2^{m,b}(M,K,\xi)$ does not depend on the realization of $D$. 
\end{lemma}

This result allows us to define the bracket $[(M,K,\xi);D]$, for a realizable Y-diagram $D$, as the class of $[(M,K,\xi);\Gamma]$ modulo $\F_2^{m,b}(M,K,\xi)$ 
for any realization $\Gamma$ of $D$. 

\begin{lemma} \label{lemmaYrel}
  Let $(M,K,\xi)\in\Ens^m(\Al,\bl)$. Let $k,k'\in\Z$. Assume the Y-diagrams $D$, $D_1$ and $D_2$ represented in Figure \ref{figYrel} 
are realizable in $(M,K,\xi)$. Then $D_0$, $D_h$ and $D_a$ are realizable in $(M,K,\xi)$, 
\begin{figure}[htb]
\begin{center}
 \diag{$\beta_1$}{$\beta_2$}{$\beta_3$}{$f_{12}$}{$f_{13}$}{$f_{23}$}{0.5} \hspace{1cm}
 \diag{$t\beta_1$}{$t\beta_2$}{$t\beta_3$}{$f_{12}$}{$f_{13}$}{$f_{23}$}{0.5} \hspace{1cm}
 \diag{$\beta_2$}{$\beta_1$}{$\beta_3$}{$f_{12}(t^{-1})$}{$f_{13}$}{$f_{23}$}{0.5}
\begin{tikzpicture}
 \draw (-4.2,-1) node{$D$} (0,-1) node{$D_h$} (4.2,-1) node{$D_a$};
\end{tikzpicture} \\ \ \\ 
 \hspace{-0.3cm}\diag{$k\beta_1+k'\beta_1'$}{$\beta_2$}{$\beta_3$}
{$\hspace{-0.8cm}kf_{12}+k'f_{12}'$}{$\hspace{0.8cm}kf_{13}+k'f_{13}'$}{$f_{23}$}{0.5} \hspace{0.1cm}
 \diag{$\beta_1$}{$\beta_2$}{$\beta_3$}{$f_{12}$}{$f_{13}$}{$f_{23}$}{0.5} \hspace{1.4cm}
 \diag{$\beta_1'$}{$\beta_2$}{$\beta_3$}{$f_{12}'$}{$f_{13}'$}{$f_{23}$}{0.5}
\begin{tikzpicture}
 \draw (-4.2,-1) node{$D_0$} (0,-1) node{$D_1$} (4.2,-1) node{$D_2$};
\end{tikzpicture} \caption{Y-diagrams} \label{figYrel}
\end{center}
\end{figure}
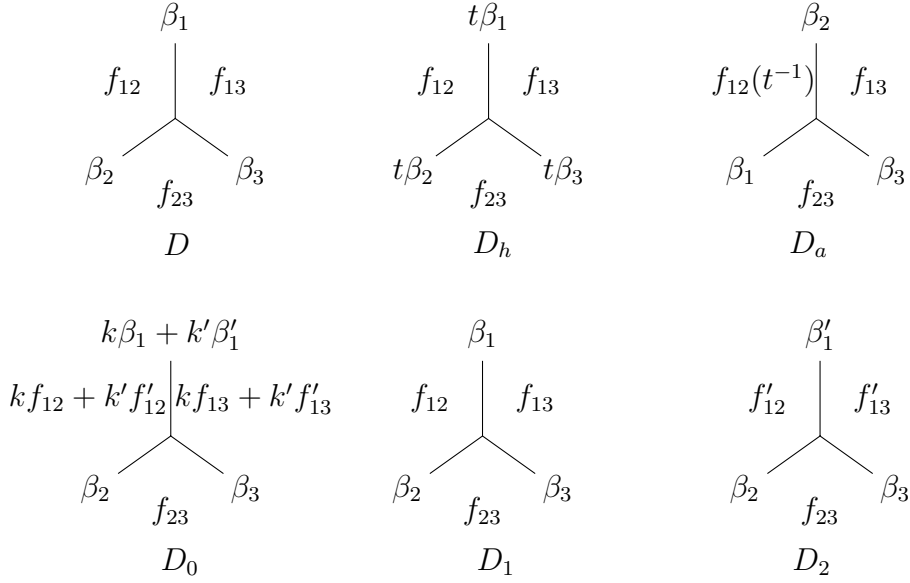
and the following relations hold:
\begin{equation}
 [(M,K,\xi);D]=[(M,K,\xi);D_h] \tag{Hol} 
\end{equation}
\begin{equation}
 \tag{AS} [(M,K,\xi);D]+[(M,K,\xi);D_a]=0
\end{equation}
\begin{equation}
 \tag{LV} [(M,K,\xi);D_0]=k\,[(M,K,\xi);D_1]+k'\,[(M,K,\xi);D_2]
\end{equation}
\end{lemma}
\begin{proof}
 Relation (Hol) is obtained by letting the internal vertex of a realization $\Gamma$ of $D$ turn once around the knot $K$. 
Relation (AS) follows from \cite[Corollary 4.6]{GGP}.
Relation (LV) follows from \cite[Chap 6, Lemma 2.10]{Mt}. 
\end{proof}

\begin{lemma} \label{lemmatriv}
 If $D=$\raisebox{-5ex}{\diag{$\beta_1$}{$0$}{$0$}{$0$}{$0$}{$f_{23}$}{0.4}} is a Y-diagram realizable in $(M,K,\xi)$, 
then $$[(M,K,\xi);D]=0.$$
\end{lemma}
\begin{proof}
 For $k\in\Z$, set $D_k=$\raisebox{-5ex}{\diag{$t^k\beta_1$}{$0$}{$0$}{$0$}{$0$}{$f_{23}$}{0.35}}. 
Set $D_{triv}=$\raisebox{-5ex}{\diag{$0$}{$0$}{$0$}{$0$}{$0$}{$f_{23}$}{0.35}}. 
Let $\delta(t)=\sum_{k\in\Z}a_kt^k$ be the annihilator of $\Al$, normalised with integral coefficients. 
By Relation (LV), $$\sum_{k\in\Z}a_k[(M,K,\xi);D_k]=[(M,K,\xi);D_{triv}]=0.$$ 

Moreover, Relation (Hol) implies $[(M,K,\xi);D]=[(M,K,\xi);D_k]$ for all $k\in\Z$. Finally:
$$\delta(1)[(M,K,\xi);D]=\sum_{k\in\Z}a_k[(M,K,\xi);D_k]=0.$$
This concludes since $\delta(1)\neq0$.
\end{proof}

\begin{lemma}
 Let $D=$\raisebox{-5ex}{\diag{$\beta_1$}{$\beta_2$}{$\beta_3$}{$f_{12}$}{$f_{13}$}{$f_{23}$}{0.4}} be a Y-diagram realizable in $(M,K,\xi)$. 
 Then the bracket $[(M,K,\xi);D]$ does not depend on the equivariant linking numbers $f_{ij}$. 
\end{lemma}
\begin{proof}
Set $D'=$\raisebox{-5ex}{\diag{$\beta_1$}{$\beta_2$}{$\beta_3$}{$f'_{12}$}{$f_{13}$}{$f_{23}$}{0.4}} and  
$P(t)=f'_{12}-f_{12}\in\Qt$. Thanks to Relation (LV), we can assume $P(t)\in\Zt$. 

Let $D_0$, $D_1$, $D_2$ be the Y-diagrams represented in Figure \ref{figdiag}.
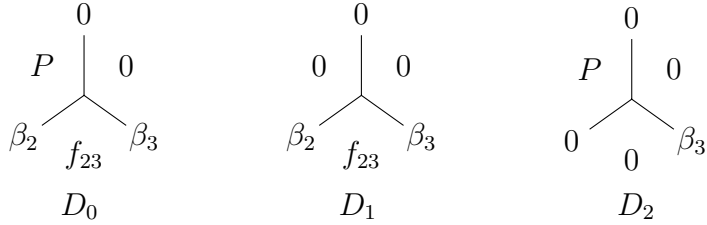
\begin{figure}[htb]
\begin{center}
 \diag{$0$}{$\beta_2$}{$\beta_3$}{$P$}{$0$}{$f_{23}$}{0.4} \hspace{1cm}
 \diag{$0$}{$\beta_2$}{$\beta_3$}{$0$}{$0$}{$f_{23}$}{0.4} \hspace{1cm}
 \diag{$0$}{$0$}{$\beta_3$}{$P$}{$0$}{$0$}{0.4}
\begin{tikzpicture}
 \draw (-3.7,-1) node{$D_0$} (0,-1) node{$D_1$} (3.7,-1) node{$D_2$};
\end{tikzpicture} \caption{Y-diagrams} \label{figdiag}
\end{center}
\end{figure}

Relation (LV) implies 
$$[(M,K,\xi);D']=[(M,K,\xi);D]+[(M,K,\xi);D_0]$$ and $$[(M,K,\xi);D_0]=[(M,K,\xi);D_1]+[(M,K,\xi);D_2].$$ 
Now, by Lemma \ref{lemmatriv}, $[(M,K,\xi);D_2]=0$, 
and by (LV), $[(M,K,\xi);D_1]=0$. 
Hence $[(M,K,\xi);D']=[(M,K,\xi);D]$, as desired.
\end{proof}

Finally, for a Y-diagram $D=$\raisebox{-5ex}{\diag{$\beta_1$}{$\beta_2$}{$\beta_3$}{$f_{12}$}{$f_{13}$}{$f_{23}$}{0.35}}, the bracket $[(M,K,\xi);D]$ 
only depends on the $\beta_i$. Hence the relation (AS) implies that this bracket only depends on 
$\beta_1\wedge\beta_2\wedge\beta_3\in\HA$. 

Now, for $\beta_1\wedge\beta_2\wedge\beta_3\in\HA$ such that each $\beta_i$ is realizable for $(M,K,\xi)$, 
we can define $[(M,K,\xi);\beta_1\wedge\beta_2\wedge\beta_3]\in\G^m_1(\Al,\bl)$ as the class of $[(M,K,\xi);D]$ in $\G^m_1(\Al,\bl)$ for any Y-diagram 
$D$ whose univalent vertices are colored by $\beta_i$ for $i=1,2,3$, with the right cyclic order. 

If $\beta_1\wedge\beta_2\wedge\beta_3$ is any tensor in $\HA$, 
there are non trivial integers $n_1$, $n_2$, $n_3$ such that $n_i\beta_i$ is realizable for $(M,K,\xi)$ for $i=1,2,3$. Set:
$$[(M,K,\xi);\beta_1\wedge\beta_2\wedge\beta_3]=\frac{1}{n_1n_2n_3}[(M,K,\xi);n_1\beta_1\wedge n_2\beta_2\wedge n_3\beta_3].$$
By (LV), this definition does not depend on the triple of integers $(n_1,n_2,n_3)$.

We finally obtain a well-defined $\Q$-linear map $\varphi : \HA \to \G^m_1(\Al,\bl)$ defined by 
$\varphi(\beta_1\wedge\beta_2\wedge\beta_3)=[(M,K,\xi);\beta_1\wedge\beta_2\wedge\beta_3]$. 
The next lemma shows that this map is canonical. 

\begin{lemma} \label{lemmacan}
 Let $(M,K,\xi)$ and $(M',K',\xi')$ be marked $\Q$SK-pairs in $\Ens^m(\Al,\bl)$. Let $\beta_1\wedge\beta_2\wedge\beta_3\in\HA$.  
Then $[(M',K',\xi');\beta_1\wedge\beta_2\wedge\beta_3]=[(M,K,\xi);\beta_1\wedge\beta_2\wedge\beta_3]$.
\end{lemma}
\begin{proof}
  Set $\zeta=\xi'\circ\xi^{-1}$. By Theorem \ref{thM3}, $(M',K',\xi')$ can be obtained from $(M,K,\xi)$ by a finite sequence of null LP-surgeries 
which induces the isomorphism $\zeta$ (up to multiplication by a power of $t$). Assume the sequence contains a single surgery $(\frac{A'}{A})$. 
Let $\tilde{X}$ be the infinite cyclic covering associated with $(M,K)$. 
Let $n_1$, $n_2$, $n_3$ be non trivial integers such that each $n_i\beta_i$ 
is realizable by a simple closed curve in $\tilde{X}$ which does not meet the preimage of $A$. 
Let $\Gamma\subset (M\setminus K)\setminus A$ be a Y-graph null in $M\setminus K$ which realizes the Y-diagram 
$D=$\raisebox{-5ex}{\diag{$n_1\beta_1$}{$n_2\beta_2$}{$n_3\beta_3$}{$f_{12}$}{$f_{13}$}{$f_{23}$}{0.42}} for any coherent values of the $f_{ij}$. 
Then:
$$[(M,K,\xi);\Gamma,\frac{A'}{A}]=[(M,K,\xi);\Gamma]-[(M',K',\xi');\Gamma].$$
In $(M',K',\xi')$, $\Gamma$ still realizes $D$. 

The case of several surgeries easily follows.
\end{proof}

\proofof{Proposition \ref{propborro}}
It is easy to see that $\varphi(\HA)=\G^{m,b}_1(\Al,\bl)$. 
So we have a surjective $\Q$-linear map $\HA\twoheadrightarrow\G^{m,b}_1(\Al,\bl)$. Now, the map $\hs:\Ens^m(\Al,\bl)\to\HA$ defines 
a $\Q$-linear map $\tilde{\hs}:\F^m_0(\Al,\bl)\to\HA$. The restriction of $\tilde{\hs}$ to $\F^{m,b}_1(\Al,\bl)$ 
is surjective. The proof of this claim is exactly the proof of Lemma \ref{lemmasurj} without the two first lines. 
By Proposition \ref{propdegree1}, $\tilde{\hs}$ induces a surjective map $\G^{m,b}_1(\Al,\bl)\twoheadrightarrow\HA$, still 
denoted by $\tilde{\hs}$. Since $\HA$ has a finite dimension, $\varphi:\HA\to\G^{m,b}_1(\Al,\bl)$ and $\tilde{\hs}:\G^{m,b}_1(\Al,\bl)\to\HA$ 
are isomorphisms. 
\fin

\subsection{Degree one invariants of marked $\Z$SK-pairs}

In this subsection, we prove Theorem \ref{thG1Z}, following the proof of Proposition \ref{propborro}. 

Fix an integral Blanchfield module $(\Al_\Z,\bl)$. Let $(M,K,\xi)\in\Ens^m_\Z(\Al_\Z,\bl)$. Since the space $\F^{m,b}_2(M,K,\xi)$ 
is a subspace of $\F^{m,\Z}_2(\Al_\Z,\bl)$, one can define 
${[(M,K,\xi);\beta_1\wedge\beta_2\wedge\beta_3]}\in\G^{m,\Z}_1(\Al_\Z,\bl)$ 
for $\beta_1\wedge\beta_2\wedge\beta_3\in\HA$ as in the previous subsection. Once again, this does not depend on the chosen marked 
$\Z$SK-pair. The only difference in the proof of Lemma~\ref{lemmacan} is that we apply Theorem \ref{thM3Z} and we use integral null 
LP-surgeries. Hence we have a well-defined, canonical and surjective map $\varphi^\Z : \HA \twoheadrightarrow \G^{m,\Z}_1(\Al_\Z,\bl)$ 
defined by $\varphi^\Z(\beta_1\wedge\beta_2\wedge\beta_3)=[(M,K,\xi);\beta_1\wedge\beta_2\wedge\beta_3]$ for any 
$(M,K,\xi)\in\Ens^m_\Z(\Al_\Z,\bl)$. 

\vspace{2ex}

\proofof{Theorem \ref{thG1Z}}
We have a surjective map $\varphi^\Z : \HA \twoheadrightarrow \G^{m,\Z}_1(\Al_\Z,\bl)$ between finite dimensional vector spaces. 
As in the proof of Proposition \ref{propborro}, we want to prove that the map $\hs$ of Proposition \ref{propdegree1} 
provides a surjective map from $\G^{m,\Z}_1(\Al_\Z,\bl)$ onto $\HA$. We must be more careful in this case, since the proof of 
the surjectivity of $\hs$ in Lemma \ref{lemmasurj} makes use of $d$-surgeries in the application of Proposition \ref{propreal}. 
These $d$-surgeries do not preserve the homology with integral coefficients of the manifold $M$, hence they do not define a move 
on the set of marked $\Z$SK-pairs. So $\hs(\Ens^m_\Z(\Al_\Z,\bl))$ may not be the whole $\HA$, but, rereading the proof of 
Lemma \ref{lemmasurj}, one easily sees that $\hs(\Ens^m_\Z(\Al_\Z,\bl))\subset\HA$ generates $\HA$ as a $\Q$-vector space. 
Hence $\hs$ induces a surjective $\Q$-linear map $\F^{m,\Z}_0(\Al_\Z,\bl)\twoheadrightarrow\HA$, which 
provides a surjective $\Q$-linear map $\tilde{\hs}:\G^{m,\Z}_1(\Al_\Z,\bl)\twoheadrightarrow\HA$. Finally, $\varphi^\Z$ and $\tilde{\hs}$ are 
isomorphisms.
\fin

\subsection{Description of $\G_1^m(\Al,\bl)$} \label{subsecG1}

In this subsection, we prove the following result, which, together with Proposition \ref{propborro}, implies Theorem \ref{thG1}. 
Fix a Blanchfield module $(\Al,\bl)$. 
\begin{proposition} \label{propG1}
 Let $(M,K,\xi)\in\Ens^m(\Al,\bl)$. For $p$ prime, let $B_p$ be a rational homology ball such that $H_1(B_p;\Z)\cong\Z/p\Z$. Then:
$$\G^m_1(\Al,\bl)\cong\left(\bigoplus_{p\textrm{ prime}} \Q\,[(M,K,\xi);\frac{B_p}{B^3}]\right)\bigoplus\G^{m,b}_1(\Al,\bl).$$
\end{proposition}

The invariants $\nu_p$ defined in Subsection \ref{subsecgrad} satisfy $\nu_p([(M,K,\xi);\frac{B_q}{B^3}])=\delta_{pq}$, 
where $\delta_{pq}$ is the Kronecker symbol. Hence $\bigoplus_{p\textrm{ prime}} \Q\,[(M,K,\xi);\frac{B_p}{B^3}]$ is indeed 
a direct sum. Note that $[(M,K,\xi);\frac{B_p}{B^3}]\in\G^m_1(\Al,\bl)$ does not depend on the marked $\Q$SK-pair $(M,K,\xi)$, thanks to 
Theorem \ref{thM3} and:
$$[(M,K,\xi);\frac{B_p}{B^3}]-[(M,K,\xi)(\frac{A'}{A});\frac{B_p}{B^3}]=[(M,K,\xi);\frac{B_p}{B^3},\frac{A'}{A}].$$
Set $\K=\bigoplus_{p\textrm{ prime}} \Q [(M,K,\xi);\frac{B_p}{B^3}])\subset\G^m_1(\Al,\bl)$. 

\begin{lemma}
 $\K\cap\G^{m,b}_1(\Al,\bl)=0$
\end{lemma}
\begin{proof}
 Since borromean surgeries preserve the homology, the invariants $\nu_p$ are trivial on $\G^{m,b}_1(\Al,\bl)$. 
Let $G\in\K\cap\G^{m,b}_1(\Al,\bl)$. On the one hand, $G$ is a linear combination of the $[(M,K,\xi);\frac{B_p}{B^3}]$, and on the other 
hand, $\nu_p(G)=0$ for all prime integer $p$. Hence $G=0$.
\end{proof}

It remains to prove that $\K\oplus\G^{m,b}_1(\Al,\bl)$ is the whole $\G^m_1(\Al,\bl)$. We first reduce the set of generators 
of $\G^m_1(\Al,\bl)$, using results from \cite{M2}. Recall $d$-surgeries were defined in Subsection \ref{subsecreal}.

\begin{definition}
 An \emph{elementary surgery} is an LP-surgery among the following ones:
\begin{enumerate}
 \item connected sum (genus 0),
 \item $d$-surgery (genus 1),
 \item borromean surgery (genus 3).
\end{enumerate}
\end{definition}
\begin{theorem}[\cite{M2} Theorem 1.15] \label{thdecsur}
 If $A$ and $B$ are two $\Q$HH's with LP-identified boundaries, then $B$ can be obtained from $A$ 
by a finite sequence of elementary surgeries and their inverses in the interior of the $\Q$HH's.
\end{theorem}

\begin{corollary} \label{cordecel}
 The space $\F^m_1(\Al,\bl)$ is generated by the $[(M,K,\xi);\frac{E'}{E}]$ where $(M,K,\xi)\in\Ens^m(\Al,\bl)$ 
and $(\frac{E'}{E})$ is an elementary null LP-surgery.
\end{corollary}
\begin{proof}
Let $[(M,K,\xi);\frac{A'}{A}]\in\F^m_1(\Al,\bl)$. By Theorem \ref{thdecsur}, $A$ and $A'$ 
can be obtained from one another by a finite sequence of elementary surgeries and their inverses. 
Write $A'=A(\frac{E_1'}{E_1})\dots(\frac{E_k'}{E_k})$. For $0\leq j\leq k$, set $A_j=A(\frac{E_1'}{E_1})\dots(\frac{E_j'}{E_j})$. 
Then:
$$[(M,K);\frac{A'}{A})]=\sum_{j=1}^k [(M,K)(\frac{A_{j-1}}{A_0});\frac{E_j'}{E_j}].$$
Conclude with $\displaystyle [(M,K);\frac{E'}{E}]=-[(M,K)(\frac{E'}{E});\frac{E}{E'}].$
\end{proof}

Let $\F_0^{\textrm{$\Q$HS}}$ be the rational vector space generated by all $\Q$HS's up to orientation-pre\-ser\-ving homeomorphism. 
Let $(\F_n^{\textrm{$\Q$HS}})_{n\in\N}$ be the filtration of $\F_0^{\textrm{$\Q$HS}}$ defined by LP-surgeries. 
Let $\displaystyle \G_n^{\textrm{$\Q$HS}}=\frac{\F_n^{\textrm{$\Q$HS}}}{\F_{n+1}^{\textrm{$\Q$HS}}}$ be the associated quotients. 
\begin{lemma}[\cite{M2} Proposition 1.8] \label{lemma07}
 For each prime integer $p$, let $B_p$ be a rational homology ball such that $H_1(B_p;\Z)\cong\frac{\Z}{p\Z}$. 
Then $\displaystyle \G_1^{\textrm{$\Q$HS}}=\bigoplus_{p\textrm{ prime}} \Q [S^3;\frac{B_p}{B^3}]$.
\end{lemma}

\begin{lemma} \label{lemmagenus0}
 For each prime integer $p$, let $B_p$ be a rational homology ball such that $H_1(B_p;\Z)\cong\frac{\Z}{p\Z}$. 
Let $(M,K,\xi)\in\Ens^m(\Al,\bl)$. Let $B$ be a rational homology ball. 
Then $\displaystyle [(M,K,\xi);\frac{B}{B^3}]$ is a rational linear combination of the 
$\displaystyle [(M,K,\xi);\frac{B_p}{B^3}]$ and elements of $\F^m_2(\Al,\bl)$.
\end{lemma}
\begin{proof}
By Lemma \ref{lemma07}, there is a relation:
$$[S^3;\frac{B}{B^3}]=\sum_{p\textrm{ prime}} a_p [S^3;\frac{B_p}{B^3}] +\sum_{j\in J} b_j [N_j;\frac{C_j'}{C_j},\frac{D_j'}{D_j}],$$
where $J$ is a finite set, the $a_p$ and $b_j$ are rational numbers, the $a_p$ are all trivial except a finite number, 
and for $j\in J$, $[N_j;\frac{C_j'}{C_j},\frac{D_j'}{D_j}]\in\F_2^{\textrm{$\Q$HS}}$. 
Make the connected sum of each $\Q$HS in the relation with $M$. We obtain:
$$[(M,K,\xi);\frac{B}{B^3}]=\sum_{p\textrm{ prime}} a_p [(M,K,\xi);\frac{B_p}{B^3}] 
+\sum_{j\in J} b_j [(M\sharp N_j,K,\xi);\frac{C_j'}{C_j},\frac{D_j'}{D_j}].$$
\end{proof}

To conclude the proof of Proposition \ref{propG1}, we need the following result about degree 1 invariants of framed rational homology tori 
(see \cite[\S 5.1]{M2} for a definition):
\begin{lemma}[\cite{M2} Corollary 5.10] \label{lemmainvtori}
For all prime integer $p$, let $M_p$ be a $\Q$HS such that $H_1(M_p;\Z)\cong\Z/p\Z$. 
Let $T_0$ be a framed standard torus. 
If $\mu$ is a degree 1 invariant of framed rational homology tori, such that $\mu(T_0)=0$ and $\mu(T_0\sharp M_p)=0$ 
for any prime integer $p$, then $\mu=0$.
\end{lemma}

\proofof{Proposition \ref{propG1}}
Let $\lambda\in(\F^m_1(\Al,\bl))^*$ such that $\lambda(\F^m_2(\Al,\bl))=0$. Assume $\lambda(\K\oplus\G^{m,b}_1(\Al,\bl))=0$. 
Let us prove that $\lambda=0$. 
Thanks to Corollary \ref{cordecel}, it suffices to prove that $\lambda$ vanishes on the brackets defined by elementary surgeries. 
It is clear for elementary surgeries of genus 3, and for elementary surgeries of genus 0, it follows from Lemma \ref{lemmagenus0}. 

Consider a bracket $[(M,K,\xi);\frac{T_d}{T_0})]$, where $(M,K,\xi)\in\Ens^m(\Al,\bl)$, 
$T_0$ is a standard torus null in $M\setminus K$, and $T_d$ is a $d$-torus for some positive integer $d$. 
Fix a parallel of $T_0$. If $T$ is a framed rational homology torus, set 
$\displaystyle \bar{\lambda}(T)=\lambda\left([(M,K,\xi);\frac{T}{T_0}]\right),$
where the LP-identification $\partial T\cong\partial T_0$ identifies the prefered parallels. Then $\bar{\lambda}$ is a degree~$1$ 
invariant of framed rational homology tori:
$$\bar{\lambda}\left( [T;\frac{A'_1}{A_1},\frac{A'_2}{A_2}]\right) = 
-\lambda\left([(M,K,\xi)(\frac{T}{T_0});\frac{A'_1}{A_1},\frac{A'_2}{A_2}]\right) =0.$$
Moreover, we have $\bar{\lambda}(T_0)=0$, and $\bar{\lambda}(T_0(\frac{B_p}{B^3}))=0$. 
Thus, by Lemma \ref{lemmainvtori}, $\bar{\lambda}=0$.
\fin

\def\cprime{$'$}
\providecommand{\bysame}{\leavevmode ---\ }
\providecommand{\og}{``}
\providecommand{\fg}{''}
\providecommand{\smfandname}{\&}
\providecommand{\smfedsname}{\'eds.}
\providecommand{\smfedname}{\'ed.}
\providecommand{\smfmastersthesisname}{M\'emoire}
\providecommand{\smfphdthesisname}{Th\`ese}

\end{document}